\documentclass{amsart}

\usepackage{color}
\usepackage{amsmath}
\usepackage{amsthm}
\usepackage{amssymb}
\usepackage{graphicx}
\usepackage{booktabs}
\usepackage{enumitem}
\usepackage{hyperref}
\usepackage{cleveref}

\usepackage{caption} 
\graphicspath{{Figures/}}

\usepackage{todonotes}

\newcommand{\cacher}[1]{}

\newcommand{\tdef}[1]{\textcolor{blue}{\emph{#1}}}

\def\mm{\mathrm{m}}

\newcommand{\blossoming}{\mathcal{B}}
\newcommand{\blossomedge}{\mathcal{E}}
\newcommand{\Tam}{\mathrm{Tam}}
\newcommand{\tamint}{\mathcal{I}}

\newcommand{\bintree}{\mathcal{T}}
\newcommand{\noncross}{NC}
\def\cS{\mathcal{S}}
\def\cvS{\widehat{\cS}}
\def\vVec{\mathrm{Vec}}

\def\infposet{\raisebox{-1.2pt}{\includegraphics[scale=0.22]{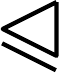}}}

\newcommand{\mdiag}{\mathcal{M\;\!\!D}}
\newcommand{\mtree}{\mathcal{M\;\!\!T}}

\newcommand{\pair}{\mathcal{X}}

\newcommand{\isep}{\mathrel{{.}\,{.}}\nobreak}

\newcommand{\can}{\operatorname{can}}
\newcommand{\mir}{\operatorname{mir}}
\newcommand{\rise}{\operatorname{rise}}
\newcommand{\dual}{\operatorname{dual}}
\newcommand{\bvec}{\mathbf{V}}
\newcommand{\dualbvec}{\bvec'} 
\newcommand{\bijbilblom}{\gamma}
\newcommand{\bijblombil}{\delta}
\newcommand{\inttoblom}{\Phi}
\newcommand{\blomtoint}{\Psi} 
\newcommand{\dne}{d_{\nearrow}}
\newcommand{\dsw}{d_{\swarrow}}
\newcommand{\treestobil}{\phi}
\newcommand{\biltotrees}{\psi}

\newcommand{\cw}{cw}
\newcommand{\ccw}{ccw}

\newcommand{\canzz}{{0 \brack 0}}
\newcommand{\canoz}{{1 \brack 0}}
\newcommand{\canoo}{{1 \brack 1}}

\def\hG{\hat{G}}

\def\oB{\overline{B}}

\def\ni{\noindent}

\newtheorem{theo}{Theorem}[section]
\newtheorem{prop}[theo]{Proposition}
\newtheorem{lemma}[theo]{Lemma}
\newtheorem{coro}[theo]{Corollary}

\theoremstyle{definition}
\newtheorem{defn}[theo]{Definition}

\theoremstyle{remark}
\newtheorem{rem}[theo]{Remark}

\begin{document}

\title{Tamari intervals and blossoming trees}

\author[Wenjie Fang, \'Eric Fusy, and Philippe Nadeau]{Wenjie Fang$^{*}$, \'Eric Fusy$^{*}$ and Philippe Nadeau$^{\dagger}$}
\thanks{$^{*}$LIGM, CNRS, Univ. Gustave Eiffel, ESIEE Paris, F-77454 Marne-la-Vallée, France. \emph{Email address}: \texttt{\{eric.fusy,wenjie.fang\}@univ-eiffel.fr}\\
$^{\dagger}$ Univ Lyon, CNRS, Université Claude Bernard Lyon 1, Institut Camille Jordan, F69622 Villeurbanne Cedex, France. 
\emph{Email address}: \texttt{nadeau@math.univ-lyon1.fr} \\
}

\begin{abstract}
 We introduce a simple bijection between Tamari intervals and the blossoming trees (Poulalhon and Schaeffer, 2006) encoding planar triangulations, using a new meandering representation of such trees. Its specializations to the families of synchronized, Kreweras, new/modern, and infinitely modern intervals give a combinatorial proof of the counting formula for each family. Compared to (Bernardi and Bonichon, 2009), our bijection behaves well with the duality of Tamari intervals, also enabling the counting of self-dual intervals.
\end{abstract}

\maketitle


\section{Introduction}

The Tamari lattice $\mathrm{Tam}_n$ is a well-known poset on Catalan objects of size $n$, that plays an important role in several domains, such as representation theory~\cite{bousquet2011number,Bergeron2012} and polyhedral combinatorics and Hopf algebra~\cite{bostan23,loday-ronco}. Motivated in part by such links, the enumeration of intervals in the Tamari lattice was first considered by Chapoton~\cite{chapoton06} who discovered the beautiful formula 
\begin{equation}\label{eq:tam}
I_n=\frac{2}{n(n+1)}\binom{4n+1}{n-1}
\end{equation}
for the number of intervals in $\mathrm{Tam}_n$. The subject has attracted much attention since then, with strikingly simple counting formulas found for several other families~\cite{bousquet2013representation,bousquet2011number,fang17}. 

As for combinatorial proofs, Bernardi and Bonichon~\cite{bernardi09} gave a bijection between Tamari intervals and planar (simple) triangulations via Schnyder woods. Then, a bijection by Poulalhon and Schaeffer~\cite{PS06} encodes the same triangulations by a class of blossoming trees, which yields~\eqref{eq:tam}. The bijection in~\cite{bernardi09} can be specialized to some subfamilies of Tamari intervals, such as Kreweras intervals~\cite{bernardi09} and synchronized Tamari intervals~\cite{fusy19}. Another strategy, for instance in~\cite{fang2018planar,fang2021bijective,fang17,pre12}, is to construct bijections between Tamari intervals and planar maps inspired by their recursive decompositions. With a similar approach, bijections between Tamari intervals and planar maps via certain branching polyominoes called fighting fish~\cite{duchi2017fighting} have been recently developed~\cite{DuHe22}. 

In this article, we present a more direct bijection denoted $\Phi$ between Tamari intervals and the blossoming trees from~\cite{PS06}. 
Our construction proceeds via certain arc-diagrams called \emph{meandering trees}. In~\Cref{sec:tamari} we show that Tamari intervals are in bijection with meandering trees, by applying simple local operations on a suitable planar representation of the pair of binary trees that form the interval. We also discuss the link between meandering trees and a tree-encoding of the interval-posets introduced by Ch\^atel and Pons~\cite{ChatelPons15}. In~\Cref{sec:blossoming} we consider the blossoming trees from~\cite{PS06} (in a bicolored version), and show that they are in bijection with meandering trees. A meandering tree directly yields a bicolored blossoming tree, by taking the unfolded version of the tree. Conversely, a bicolored blossoming tree can be turned into a meandering tree by a certain \emph{closure-mapping}, which is a variation of the closure-mapping in~\cite{PS06} that yields a rooted simple planar triangulation.

In~\Cref{sec:properties} we use our bijection to track several parameters on Tamari intervals, such as the number of entries in each of the 3 canopy-types, which yields a simple derivation of the associated trivarate generating function~\cite{fusy19} and of a recent bivariate counting formula for Tamari intervals~\cite{bostan23}. 
Due to its simplicity, our bijection is well-suited for specializations to known subfamilies of Tamari intervals, by characterizing the blossoming trees in each case (\Cref{sec:specialize}). In addition to synchronized intervals, whose specialization is much simpler than that in~\cite{fusy19}, and Kreweras intervals, already given in \cite{bernardi09}, our bijection also specializes to new/modern intervals~\cite{chapoton06,rognerud18} and infinitely modern intervals~\cite{rognerud18}. It allows us to recover the known counting formulas for these families (see~\Cref{table:self_dual}) in a uniform way, as done in~\Cref{sec:counting}. 
Compared to~\cite{bernardi09}, our bijection has also the advantage that it transfers the duality involution on Tamari intervals in a simple way, which amounts to a color-switch in blossoming trees (\Cref{lem:duality-commutation}). Self-dual intervals thus correspond to blossoming trees with a half-turn symmetry. By counting these trees, we obtain simple counting formulas for each family we consider (see~\Cref{table:self_dual}). These formulas are new to our knowledge, except for Kreweras intervals. 

The following statement summarizes our main results.
\begin{theo}
 The bijection $\inttoblom$ between intervals in $\Tam_n$ and bicolored blossoming trees of size $n$ sends self-dual intervals to blossoming trees with a half-turn symmetry. Its specializations to synchronized, Kreweras, modern/new, and infinitely modern intervals yield combinatorial proofs
of counting formulas for intervals and self-dual intervals in each case, see~\Cref{table:self_dual}.
\end{theo}

Finally, we conclude in~\Cref{sec:final} with remarks and observations related to our new bijection. In particular, we note that, besides color switch, another natural involution on blossoming trees is to apply a reflection. This yields a new involution on Tamari intervals with interesting properties, see~\Cref{sec:invol}.

\section{Tamari intervals and their meandering representation}
\label{sec:tamari}

\subsection{Tamari lattice and intervals}

We first recall the definition of the Tamari lattice, formulated here on binary trees, which are either a single leaf, denoted by $\epsilon$, or a binary node with two sub-trees $T_L, T_R$, denoted by $(T_L, T_R)$. The size of a binary tree $T$ is the number of its binary nodes. We denote by $\bintree_n$ the set of binary trees of size $n$. For $u$ a node in $T$, we denote by $T_u$ the sub-tree of $T$ rooted at $u$.

\begin{defn} \label{defn:tamari-binary-tree}
For $T,T'$ two elements in $\bintree_n$, write $T \prec T'$ if there is a node $u$ in $T$ such that $T_u$ has the form $((T_A, T_B), T_C)$, and $T'$ is obtained from $T$ by replacing $T_u$ by $(T_A, (T_B, T_C))$, the replacement operation being called right rotation. The \tdef{Tamari lattice} $\Tam_n = (\bintree_n, \leq)$ is defined as the reflexive and transitive closure of this relation. 

 A \tdef{Tamari interval} of size $n$ is a pair $(T,T')$ such that $T\leq T'$ in $\Tam_n$. The set of Tamari intervals of size $n$ is denoted $\tamint_n$.
\end{defn}

It is not immediately clear that the partially ordered set $\Tam_n$ is actually a lattice, see~\cite{huang-tamari} for a proof.

\begin{figure}[!ht]
\begin{center}
\includegraphics[width=\textwidth]{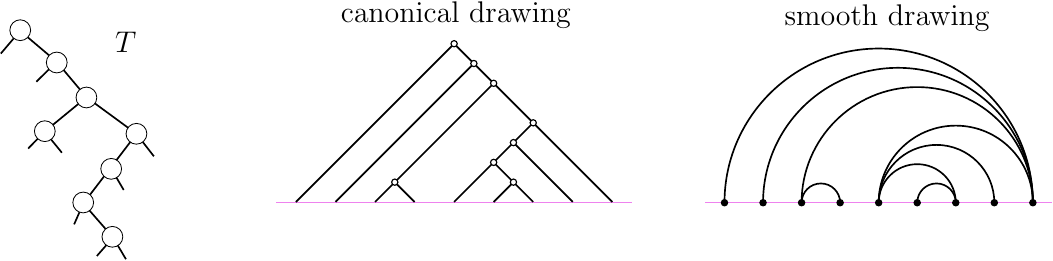}
\end{center}
\caption{A binary tree $T$ with its canonical drawing and smooth drawing.}
\label{fig:drawings}
\end{figure}

\subsection{Some representations and encodings of binary trees}\label{sec:catalan}
Throughout the article we will use the notation $[a\isep b] := \{a, a+1, \ldots, b\}$ and $[n] := [1\isep n]$. 
For $T\in\bintree_n$, the \tdef{canonical drawing} of $T$ is the crossing-free drawing of $T$ where its $n+1$ leaves from left to right are placed at the points of abscissas $0,\ldots,n$ on the $x$-axis, its nodes are in the upper half-plane, and its edges from a node to the left (resp. right) child are segments of slope $+1$ (resp. $-1$). For every node $v$ of $T$, the \tdef{wedge} of $v$ is the concatenation of the segment of slope $+1$ from $v$ to the leftmost leaf in the subtree $T_v$ rooted at $v$ and of the segment of slope $-1$ from $v$ to the rightmost leaf in $T_v$. The \tdef{smooth drawing} of $T$ is obtained by replacing every node and associated wedge by a semi-circle (in the upper half-plane) connecting the two incident leaves, see~\Cref{fig:drawings}. By construction we have the following characterization:

\begin{lemma}\label{lem:smooth}
 Smooth drawings of binary trees of size $n$ are planar arc-diagrams on integer-points of abscissa from $0$ to $n$ on the horizontal line, with all arcs in the upper half-plane,  characterized by the following properties:
 \begin{itemize}
 \item For $t \in [n]$, the unit-segment $[t-1, t]$ is below an arc. We denote by $a_t$ the unique arc covering $[t-1, t]$ and visible from it (i.e., the deepest one);
 \item Let $x_{\ell}$ be the abscissa of the left end of $a_t$. If $x_{\ell} < t - 1$, then there is an arc from the point at $x_{\ell}$ to the point at $t - 1$;
 \item Let $x_{r}$ be the abscissa of the right end of $a_t$. Similarly, if $t < x_r$ then there is an arc from the point at $t$ to the point at $x_r$. 
 \end{itemize}
\end{lemma}

\begin{rem}\label{rem:one-to-one}
 The mapping from $[t-1,t]$ to $a_t$ in~\Cref{lem:smooth} clearly gives a 1-to-1 correspondence between the $n$ unit-segments and the $n$ arcs.
\end{rem}
\begin{rem}
 A smooth drawing of a binary tree also corresponds to an \emph{alternating layout} of a plane tree with $n$ edges, where \emph{alternating} means that all neighbours of a vertex are on the same side, either all to the left or all to the right.
\end{rem}

\begin{figure}[!ht]
\begin{center}
\includegraphics[width=\textwidth]{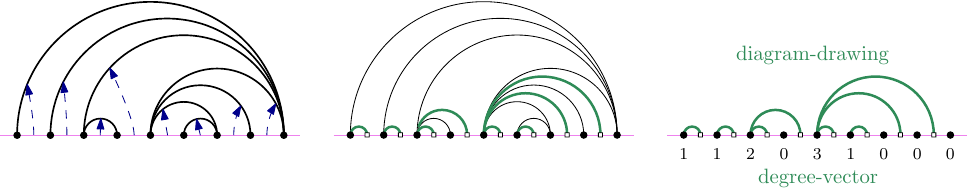}
\end{center}
\caption{Construction of the diagram-drawing (with the degree-vector indicated below) from the smooth drawing.}
\label{fig:diagramdrawing}
\end{figure}

\begin{defn} \label{defn:tree-diagram-drawing}
 A \tdef{diagram-drawing} of size $n$ is a non-crossing arc-diagram with $2n + 1$ points at abscissas $0, 1/2, 1, \ldots, n - 1/2, n$ on the $x$-axis, with the $n + 1$ integral points colored black, and the others white, such that all arcs are in the upper half-plane and have a black point as left  end and a white point as right end, with each white point incident to a single arc.

 For $T\in\bintree_n$, the \tdef{diagram-drawing} $\hat{T}$ of $T$ is the diagram-drawing obtained from the smooth drawing of $T$ by converting each arc $a_t$, as defined in~\Cref{lem:smooth}, into an arc from the white point at $t - 1/2$ to the black point at the left end of $a_t$, see~\Cref{fig:diagramdrawing}.
\end{defn}

\begin{prop}
The mapping $T\to\hat{T}$ is a bijection from $\bintree_n$ to diagram-drawings of size $n$.
\end{prop}
\begin{proof}
The mapping to recover the smooth drawing of $T$ from its diagram-drawing $\hat{T}$ is as follows: for each white point $w$ of $\hat{T}$, define the \tdef{right-attachment point} of $w$ as the rightmost black point $b$ that can be reached from $w$ by traveling in the upper half-plane without crossing an arc, i.e., $b$ is the black point at $x=n$ if there is no arc above $w$, and if $w$ is covered by an arc $b'\to w'$ then $b$ is the black point to the left of $w'$. Then, for each arc $b\to w$ in $\hat{T}$, the corresponding arc in the smooth drawing of $T$ connects $b$ to the right-attachment point of $w$. It is easy to check that, starting from any diagram-drawing, this mapping yields a valid smooth drawing (satisfying the conditions of~\Cref{lem:smooth}), and that it is the inverse of the mapping $T\to\hat{T}$.
\end{proof}

\begin{defn}
A \tdef{degree-vector} of size $n$ is a vector $(d_0,\ldots,d_{n})\in\mathbb{N}^{n+1}$ satisfying $\sum_{j=0}^{n} d_j =n$, and $\sum_{j=0}^i d_j >i$ for each $i\in[0\isep n-1]$.  
Equivalently, the sequence of steps $(d_i-1)_{0\leq i\leq n}$ gives a Łukasiewicz walk, whose definition can be found in e.g.~\cite[Section~I.5]{flajolet}.
\end{defn}

For $T \in \bintree_n$, the \tdef{degree-vector} of $T$, denoted by $\dne(T)$, is the vector $(d_0,\ldots,d_{n})$ such that  $d_i$ is the number of arcs incident to the black point at $x = i$ in the diagram-drawing of $T$, for each $i\in [0\isep n]$. Equivalently, by the definition of smooth drawing of binary trees, $\dne(T)_i$ is the right-degree of the black vertex at $x = i$ in the smooth drawing of $T$, and is the number of nodes on the maximal left branch of $T$ ending at the leaf at $x = i$ (thus $0$ for a right leaf); see~\Cref{fig:diagramdrawing}, right. This correspondence is a classical bijection between degree-vectors of size $n$ and $\bintree_n$. More precisely, the left-sibling-right-child bijection converts a binary tree into a rooted plane tree, in which nodes on a maximal left branch become children of the same node. Then the Łukasiewicz walk of the rooted plane tree is the sequence formed by arity of each node minus 1, with the nodes ordered by left-to-right DFS traversal.

\begin{figure}[!ht]
\begin{center}
\includegraphics[width=0.86\textwidth]{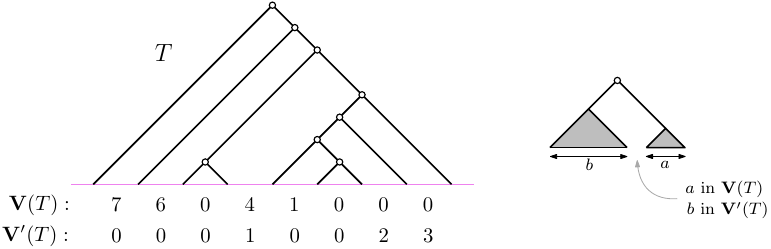}
\end{center}
\caption{A binary tree $T$, its bracket-vector $\bvec(T)$ and dual bracket-vector $\dualbvec(T)$.}
\label{fig:bracket}
\end{figure}

Finally, we recall the bracket-vector and dual bracket-vector\footnote{Bracket-vectors, and similarly dual bracket-vectors, 
are specified by inequality constraints which we do not reproduce here, see~\cite{huang-tamari}.} encoding of a binary tree $T\in\bintree_n$. 
We label nodes of $T$ from $1$ to $n$ by infix order, with $v_i$ the node of label $i$. Let $a_i$ (resp. $b_i)$ be the size of the right (resp. left) subtree of $v_i$. The \tdef{bracket-vector} of $T$ is defined as $\bvec(T) = (a_1, \ldots, a_n)$, and the \tdef{dual bracket-vector} of $T$ as $\dualbvec(T) = (b_1, \ldots, b_n)$. See~\Cref{fig:bracket} for an illustration. The bracket-vector encoding is convenient to characterize Tamari intervals. For $T, T'\in\bintree_n$, it is known~\cite{huang-tamari} that $(T,T')\in\tamint_n$ if and only if $\bvec(T)_i \leq \bvec(T')_i$ for all $i$, or equivalently if and only if $\dualbvec(T)_i \geq \dualbvec(T')_i$ for all $i$. 

\begin{rem}\label{rem:dualbracket}
The dual bracket-vector is closely related to the diagram-drawing: for $T \in \bintree_n$ and for $t \in [n]$, the unique arc in $\hat{T}$ incident to the white point at $x=t - 1/2$ is connected to the black point at $x = t - 1 - b_t$.
\end{rem}

\subsection{From pairs of binary trees to meandering diagrams/trees}\label{sec:phi}

\begin{figure}[!ht]
\begin{center}
\includegraphics[width=\textwidth]{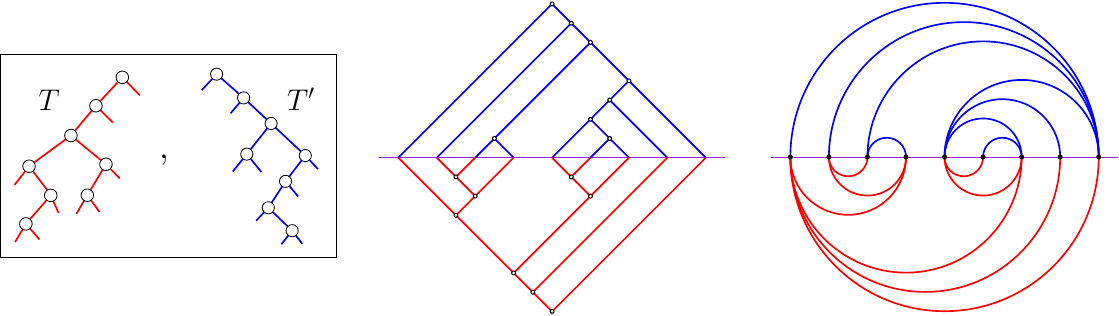}
\end{center}
\caption{A pair $(T,T')$ of binary trees of same size, its canonical drawing, and its smooth drawing.}
\label{fig:drawingpair}
\end{figure}

The \tdef{mirror} of a binary tree $T$, denoted by $\mir(T)$, is the mirror image of $T$ exchanging left and right. The \tdef{mirror canonical drawing} (resp. \tdef{mirror smooth drawing}) of $T$ is the canonical drawing (resp. smooth drawing) of $\mir(T)$ rotated by a half-turn, which preserves the left-to-right order of leaves of $T$. Equivalently, the mirror canonical (resp. smooth) drawing is the image of the canonical (resp. smooth) drawing of $T$ by the mirror exchanging up and down. 

Let $\pair_n:=\bintree_n\times \bintree_n$. For $X = (T,T') \in \pair_n$, the \tdef{canonical drawing} (resp. \tdef{smooth drawing}) of $X$ is the superimposition of the canonical (resp. smooth) drawing of $T'$ with the mirror canonical (resp. smooth) drawing of $T$, see~\Cref{fig:drawingpair}. In this case, the \tdef{upper diagram-drawing} of $X$ is the diagram-drawing of $T'$, while the \tdef{lower diagram-drawing} of $X$ is the diagram-drawing of $\mir(T)$ rotated by a half-turn. 
The \tdef{diagram-drawing} of $X$ is the superimposition of the upper and lower diagram-drawing of $X$. As a convention, in each of the 3 representations of $X$, the arcs are blue (resp. red) in the upper (resp. lower) part. Let \tdef{$\treestobil$} be the mapping that sends $X\in\pair_n$ to its diagram-drawing;  see~\Cref{fig:local_op_phi} for the local construction, and~\Cref{fig:examples_phi} for two examples. 

\begin{defn}\label{defn:trees-to-bilabel}
 A \tdef{meandering diagram} $M$ of size $n$ is a non-crossing arc-diagram  with $2n + 1$ points, at $0, \tfrac{1}{2}, 1, \ldots, n-\tfrac{1}{2}, n$ on the $x$-axis, colored black for integral points and white for half-integral ones, with all upper (resp. lower) arcs having a black (resp. white) left end and a white (resp. black) right end. 
 
 The \tdef{underlying graph} of $M$ is the graph with black points as vertices, and edges indexed by white points, relating the black endpoints of its incident upper and lower arcs. A \tdef{meandering tree} is a meandering diagram whose underlying graph is a tree. Let $\mdiag_n$ (resp. $\mtree_n$) be the set of meandering diagrams (resp. meandering trees) of size $n$.
\end{defn}

\begin{figure}
\begin{center}
\includegraphics[width=0.8\textwidth]{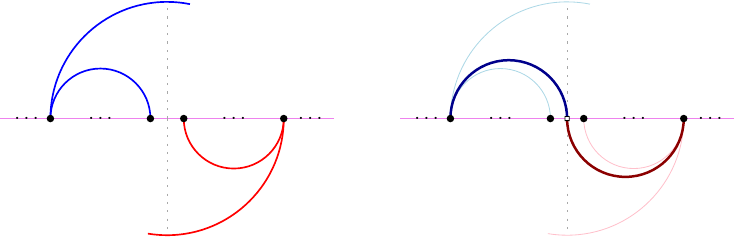}
\end{center}
\caption{The operation of $\treestobil$ on the smooth drawing of a pair $(T,T') \in \pair_n$, performed at each segment between consecutive black points on the horizontal axis. The smaller blue arc and red arc in the left figure may be reduced to a point.}
\label{fig:local_op_phi}
\end{figure}

\begin{figure}
\begin{center}
\includegraphics[width=\linewidth]{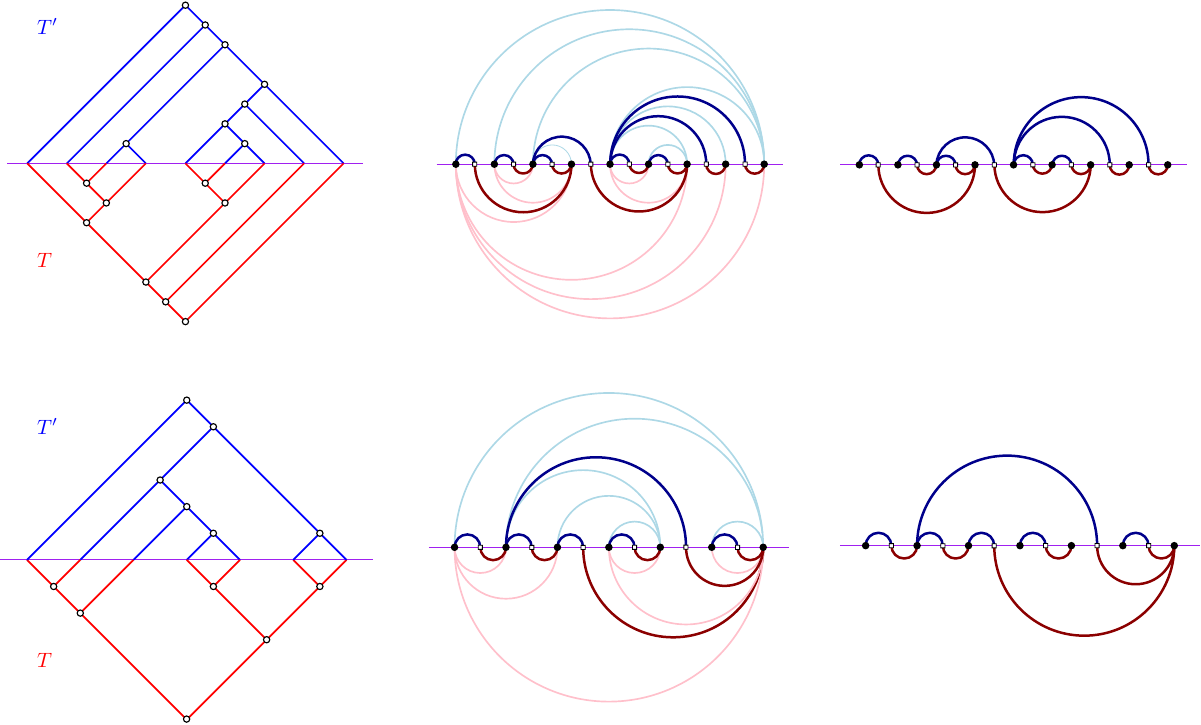}
\end{center}
\caption{The top (resp. bottom) row shows a pair $X=(T,T')\in \pair_n$ that is (resp. is not) in $\tamint_n$, and the corresponding meandering diagram $\phi(X)$ that is (resp. is not) a meandering tree.}
\label{fig:examples_phi}
\end{figure}

\begin{prop}\label{prop:phi}
For $n \geq 1$, the mapping $\treestobil$ is a bijection between $\pair_n$ and $\mdiag_n$. It specializes to a bijection between $\tamint_n$ and $\mtree_n$. 
\end{prop}
The proof is given next in Section~\ref{sec:proofphi}. 

\begin{rem}
By~\Cref{rem:dualbracket}, the mapping $\phi$ also has a simple formulation in terms of bracket-vector and dual bracket-vector. For $X=(T, T') \in \pair_n$, with $\bvec(T) = (a_1, \ldots, a_n)$, and $\dualbvec(T') = (b_1, \ldots, b_n)$, $\treestobil(X)$ is given by its lower arcs $(t - \tfrac{1}{2}, t+a_t)$ and upper arcs $(t - \tfrac{1}{2}, t-b_t-1)$ for all $t \in [n]$.
\end{rem}

\subsection{Proof of Proposition~\ref{prop:phi}}\label{sec:proofphi}

The inverse $\biltotrees$ of the mapping $\treestobil$ relies on the equivalence between the representations of Catalan structures discussed in~\Cref{sec:catalan}. For $M \in \mdiag_n$, we consider the diagram-drawing made by the upper half-plane part, from which we compute (via right-attachment points) the corresponding smooth drawing, and turn it into the canonical drawing of a binary tree $T'\in\bintree_n$. We do the same for the half-turn of the lower diagram-drawing, yielding a binary tree of size $n$, whose mirror is denoted $T$. Then $\biltotrees$ is the mapping associating $(T,T')$ to $M$. 

It remains to show the specialization statement in Proposition~\ref{prop:phi}. Our proof relies on a forbidden pattern characterization of Tamari intervals and of meandering trees, 
and on the fact that the forbidden patterns are in correspondence via $\phi$.

For $(T, T') \in \pair_n$, in its canonical drawing, a \tdef{flawed pair} is a pair $v, v'$ of nodes respectively in $T$ and $T'$ such that, for $x_\ell$ and $x_r$ (resp. $x_{\ell}'$ and $x_r'$) the abscissas of the leftmost and the rightmost leaf in $T_v$ (resp. $T'_{v'}$), we have $x_{\ell}' < x_{\ell} \leq x_r' < x_r$. In the smooth drawing of $(T, T')$, a \tdef{flawed pair} is a pair made of a lower arc $a$ and an upper arc $a'$ such that, for $x_{\ell}$ and $x_r$ (resp. $x_{\ell}'$ and $x_r'$) the abscissas of the left and the right end of the lower (resp. upper) arc, we have $x_{\ell}'<x_{\ell}\leq x_r'<x_r$. Obviously a flawed pair in the canonical drawing gives a flawed pair in the smooth drawing and vice versa.
 
\begin{lemma}\label{lem:flawed_canonical}
 A pair $(T,T')\in\pair_n$ is in $\tamint_n$ if and only if it has no flawed pair in its canonical (equivalently, smooth) drawing. 
\end{lemma} 
\begin{proof}
Recall the specification of bracket-vectors illustrated in~\Cref{fig:bracket}. 
If $(T,T')\notin\tamint_n$, then there exists $i\in[n]$ such that the $i$th entry is larger in $\bvec(T)$ than in $\bvec(T')$. 
Let $v$ be the right child of the $i$th node of $T$, and let $v'$ be the $i$th node of $T'$. Then clearly $v,v'$ form a flawed pair. 

Conversely, suppose $(T,T')\in\pair_n$ has a flawed pair $v,v'$. Then, using the notation $x_{\ell},x_r,x_{\ell}',x_r'$ for the related absissas as defined before the statement of the lemma, the $x_{\ell}$-th entry is at least $x_r-x_{\ell}$ in $\bvec(T)$,
and at most $x_{r'}-x_{\ell}$ in $\bvec(T')$, hence is smaller in $\bvec(T')$ than in $\bvec(T)$, so that $(T,T')\notin \tamint_n$. 
\end{proof}

Figure~\ref{fig:examples_phi} shows examples of pairs $(T,T')$ being in $\tamint_n$ or not, one can find a flawed pair in the pair not in $\tamint_n$. For a meandering diagram $M\in\mdiag_n$, a \tdef{flawed pair} is a pair made of a lower arc $(x_\ell, x_r)$ and an upper arc $(x'_\ell, x'_r)$ such that $x_{\ell}' < x_\ell < x_r' < x_r$. 
 \begin{lemma}\label{lem:flawed_meandering}
 A meandering diagram $M$ is a meandering tree if and only if it has no flawed pair.
 \end{lemma}
 \begin{proof}
Assume $M$ has a flawed pair, with $(x_\ell, x_r)$ and $(x'_\ell, x'_r)$ its two arcs. We observe that $x_\ell$ and $x'_r$ correspond to white points, and $x_r, x'_\ell$ black ones. Consider the horizontal segment $S = [x_\ell, x'_r]$ on the $x$-axis that we call the \emph{central segment} of the flawed pair. We pick $S$ to have minimal length $x'_r-x_\ell$ over all flawed pairs. As both ends of $S$ are white points, it must contain at least one black point. Let $G$ be the underlying graph of $M$, with $V_S$ the set of vertices corresponding to black points
in $S$. We are going to show that vertices in $V_S$ are only connected in $G$ to other vertices in $V_S$. 
Let $b$ be a black point in $S$, and consider an arc starting from some white point $w$ and ending at $b$. If it is an upper arc, that is $b < w$, then $w< x'_r$ since $(b,w)$ cannot cross $(x'_\ell,x'_r)$ and the last inequality is strict due to the uniqueness of upper arc starting from a white point. Symmetrically, the same analysis holds when it is a lower arc. Thus, all black points in $S$ are linked to white points strictly inside $S$. Now let $w$ be a white point strictly inside $S$ such that its associated upper arc leaves $S$ and reaches a black point $b<x_l$. Then $(x_\ell,x_r)$ and $(b,w)$ form a flawed pair with central segment of length $w-x_\ell<x'_r-x_\ell$, contradicting the minimality of $S$. The same analysis works symmetrically if a lower arc starts from $w$. Since vertices in $V_S$ can only reach vertices in $V_S$, $G$ is disconnected, so that $M$ is not a meandering tree.

Now assume that $M$ has no flawed pair. Let $G$ be the underlying graph of $M$. Assume for contradiction that $G$ is not a tree. Since $G$ has excess $-1$, it has a cycle $\sigma$, corresponding to a cycle $\hat{\sigma}$ in $M$. Let $b$ (resp. $b'$) be the leftmost (resp. rightmost) point on $\hat{\sigma}$; note that $b$ and $b'$ are black. Let $(b, w)$ be the higher upper arc entering $b$ in $\hat{\sigma}$. As $M$ is crossing-free, all upper arcs below $(b,w)$ stay in the segment $S = [b, w]$. For lower arcs starting from points in $S$, only those starting from white points may go to the right of $S$. As $M$ has no flawed pair, every such lower arc ends inside $S$, except the lower arc incident to $w$. We thus conclude that any path in $\hat{\sigma}$ from $b$ to $b'$ must pass by this lower arc, which contradicts the fact that $\hat{\sigma}$ is a cycle. 
 \end{proof}

\begin{figure}[!ht]
\begin{center}
\includegraphics[width=10cm]{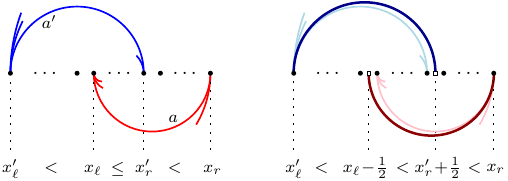}
\end{center}
\caption{The 1-to-1 correspondence between maximal flawed pairs of $X\in\pair_n$ (in its smooth drawing) and flawed pairs of $\treestobil(X)$.}
\label{fig:transfer_flawed}
\end{figure}

\begin{lemma}\label{lem:flawed_transfer}
 Let $X = (T,T') \in \pair_n$. Then $X$ has a flawed pair in its smooth drawing if and only if $\treestobil(X)$ has a flawed pair. 
\end{lemma} 
\begin{proof}
 In the smooth drawing of $X$, a flawed pair $a,a'$ is called \emph{maximal} if $a'$ is the outermost upper arc with its right end at $x_r'$, and $a$ is the outermost lower arc with its left end at $x_{\ell}$. Note that if $X$ has a flawed pair then it has a maximal one. Given~\Cref{lem:smooth}, in a maximal flawed pair, $a'$ is not the outermost arc with the left end at $x_\ell'$, and similarly $a$ is not the outermost arc with the right end at $x_r$. Then, as illustrated in~\Cref{fig:transfer_flawed}, the maximal flawed pairs of $X$ are in 1-to-1 correspondence with the flawed pairs of $\treestobil(X)$, where the correspondence preserves $x_\ell'$ and $x_r$, increases $x_r'$ by $1/2$, and decreases $x_\ell$ by~$1/2$. 

\end{proof}

The three above lemmas then ensure that the mapping $\treestobil$ specializes into 
a bijection between $\tamint_n$ and $\mtree_n$.

\begin{rem}
One can also consider specializations to other well-known subfamilies of pairs of binary trees. One promising example is given by twin pairs of binary trees, introduced in~\cite{Dulucq}, where they are shown to be equinumerous to Baxter permutations. A subfamily of these twin pairs is also known to be in bijection to synchronized Tamari intervals~\cite{fusy19}, it would be interesting to recover this bijection via the associated meandering diagrams. 	
\end{rem}
 
\subsection{Connection with previous work} To conclude the section, we discuss the connections of our construction with previous work on Tamari intervals.

\subsubsection{Relation with interval-posets} 
An \tdef{interval-poset} $P$ of size $n$ is a poset $([n], \infposet)$ such that, for any $t \in [n]$, the set $\mathrm{Int}_t:=\{s\in [n] \mid s\ \infposet\ t\}$ is an interval $[a_t\isep b_t]$.
Interval-posets were introduced by Ch\^atel and Pons~\cite{ChatelPons15}, who gave a size-preserving bijection to Tamari intervals. 

\begin{figure}
\begin{center}
\includegraphics[width=10cm]{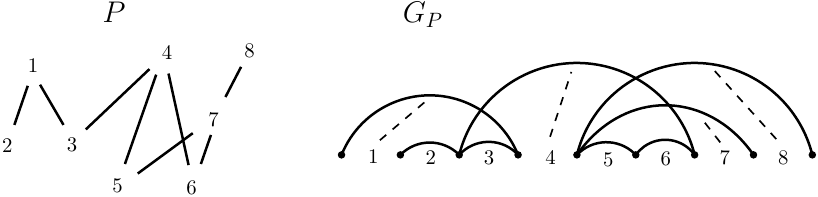}
\end{center}
\caption{An interval-poset $P$ and the associated graph $G_P$ (always a tree). The interval-poset is the one associated to the Tamari interval in the top row of~\Cref{fig:examples_phi}
by the Ch\^atel-Pons bijection.}
\label{fig:int_poset}
\end{figure}

\begin{figure}
\begin{center}
\includegraphics[width=12.4cm]{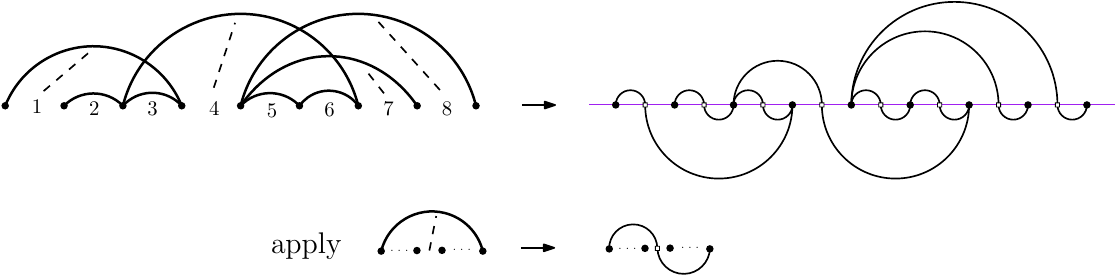}
\end{center}
\caption{Given an interval-poset $P$ associated to a Tamari interval $I$ by the Ch\^atel-Pons bijection, applying the shown local operation at each edge of $G_P$ yields $\phi(I)$.}
\label{fig:local_op_int_poset_tree}
\end{figure}

An interval-poset $P$ of size $n$ can be encoded by a graph $G_P$ with $n+1$ vertices carrying distinct labels in $[0\isep n]$ and $n$ edges carrying distinct labels in $[n]$. For each $t\in[n]$ the edge $e_t$ of label $t$ connects the vertices of labels $a_t-1$ and $b_t$. The graph $G_P$ can be represented with the vertices aligned on the $x$-axis, and edges as upper arcs, with a dashed link from the arc of $e_t$ to the unit-segment $[t-1,t]$ on the $x$-axis, see~\Cref{fig:int_poset}. The graph $G_P$ is implicitly considered in~\cite{rognerud18}, where it is shown to be a non-crossing tree if and only if $P$ is the interval-poset of a so-called exceptional Tamari interval, i.e., identifying to a Kreweras interval under a standard bijection from binary trees to non-crossing partitions (cf.~\Cref{subsub:kreweras}). 

It can be shown that, for any interval-poset $P$, the graph $G_P$ is a tree; one can argue by deletion of a min-element in $P$, and induction. Let $I=(T,T')$ be the Tamari interval associated to $P$ by the bijection in~\cite{ChatelPons15}, with the nodes of $T$ and of $T'$ labeled by left-to-right infix order. Let $t\in[n]$, with $v_t$ and $v_t'$ the corresponding nodes in $T$ and $T'$. 
 It follows from Remark~55 in~\cite{chatel2019weak}, stated as Proposition~2.1 in~\cite{rognerud18}, that $a_t$ is the label of the leftmost node in the subtree of $T'$ rooted at $v_t'$, and $b_t$ is the label of the rightmost node in the subtree of $T$ rooted at $v_t$. Hence, the meandering tree $\phi(I)$ is just obtained from $G_P$ by applying at each edge the local operation shown in~\Cref{fig:local_op_int_poset_tree}. 
 
 We actually discovered the bijection $\phi$ via the tree-encoding of interval-posets, our presentation in Section~\ref{sec:phi} is the shortcut version that operates directly on a pair of binary trees.

\subsubsection{Relation to cubic coordinates for Tamari intervals}
 The characterization of Tamari intervals given by~\Cref{lem:flawed_canonical} can be related to their encoding introduced in~\cite{combe2019cubic}. For $T\in\bintree_n$, note that the last entry of $\bvec(T)$ and the first entry of $\dualbvec(T)$ are always $0$. Moreover, for $(T,T')\in\tamint_n$, with $\bvec(T)=(a_1,\ldots,a_n)$ and $\dualbvec(T')=(b_1,\ldots,b_n)$, the absence of flawed pair ensures that, for $i \in [n-1]$, the entries $a_{i}$ and $b_{i+1}$ can not both be positive. Accordingly, the vector $(c_1,\ldots,c_{n-1})\in\mathbb{Z}^{n-1}$ defined by $c_i=-a_i$ if $a_i>0$, and $c_i=b_{i+1}$ otherwise, encodes the Tamari interval $(T,T')$. 

 This retrieves the ``cubic coordinates" vector of a Tamari interval of~\cite{combe2019cubic}. Without getting into details, this vector is characterized by certain inequalities coming from the study of interval-posets. There are three kinds of inequalities, coming from $\bvec(T)$, from $\dualbvec(T')$ and from the absence of (maximal) flawed pairs. In~\cite{combe2019cubic} the characterization was obtained using the order relation in the Tamari lattice expressed in terms of both bracket-vector and dual bracket-vector, via the construction of interval posets. However, here via~\Cref{lem:flawed_canonical}, we can obtain the same result directly, and using only the characterization of the Tamari lattice with bracket-vectors.

\section{Blossoming trees and their meandering representation}
\label{sec:blossoming}

We consider the following trees, which are in bijection with simple triangulations~\cite{PS06}.

\begin{defn} \label{defn:blossoming-tree}
 A \tdef{blossoming tree} $B$ is an unrooted plane tree such that each \tdef{node} (vertex of degree at least $2$) has exactly two neighbors that are leaves. We only consider such trees with at least two nodes. Edges incident to leaves are called \tdef{buds}, each bud being represented as an outgoing arrow. Edges not incident to leaves are called \tdef{plain edges}. The \tdef{size} $n\geq 1$ of $B$ is its number of plain edges, which is also the number of nodes minus $1$.
\end{defn}

 See the top left of Figure~\ref{fig:closure} for an illustration. The colors in that example come from the following refinement: a blossoming tree is \tdef{bicolored} if the half-edges of its plain edges are colored red or blue without monochromatic plain edge, and at each node the two buds separate the half-edges into a group of blue half-edges and a group of red half-edges  (one of the groups may be empty). We denote by $\blossoming_n$ the set of bicolored blossoming trees of size~$n$. 

\begin{rem}
 The bicoloration is unique up to the color of a starting half-edge. Hence, a blossoming tree, as an unrooted tree, yields at most two bicolored blossoming trees, and it yields just one if and only if it is stable by the half-turn symmetry.
\end{rem}

\begin{figure}[!ht]
\begin{center}
\includegraphics[width=\textwidth]{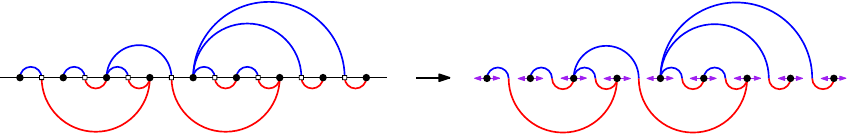}
\end{center}
\caption{A meandering tree $M$, and the corresponding bicolored blossoming tree $B=\bijbilblom(M)$.}
\label{fig:opening}
\end{figure}

\begin{defn} \label{defn:bilabel-to-blossoming} 
For $M \in \mtree_n$, we obtain a bicolored blossoming tree $B \in \blossoming_n$ by adding a ``left'' and a ``right'' bud at each black point along the $x$-axis, while keeping the colors of arcs, which are turned into half-edges of plain edges in $B$, see~\Cref{fig:opening}. Let \tdef{$\bijbilblom$} be the mapping sending $M$ to $B$. 
\end{defn}

\medskip

\ni{\bf Inverse mapping of $\bijbilblom$.} To prove that $\bijbilblom$ is a bijection, we now describe its inverse mapping $\bijblombil$. For this, we need to define the closure of a blossoming tree. It is constructed  in a different way from~\cite{PS06}, where it yields a rooted simple triangulation, cf.~\Cref{rem:poulalhon-schaeffer}. 
A \tdef{planar map} is the embedding of a connected graph in the plane such that edges only intersect at vertices. The embedding cuts the plane into \tdef{faces}, and the one extending to infinity is the \tdef{outer face}.

For simplicity, in the following, we will use the shorthand ``\cw'' (resp. ``\ccw'') for ``clockwise'' and ``counterclockwise''.

 Given a bicolored blossoming tree $B$, its \tdef{closure}, denoted by $\oB$, is constructed as follows. For each plain edge $e$, we insert an \tdef{edge-vertex} $v_e$ in its middle, and we attach new open half-edges called \tdef{legs} to $v_e$ on each side of $e$. The \ccw-contour on the current tree yields a cyclic parenthesis word, whose opening (resp. closing) parentheses are given by buds (resp. legs). We then match buds and legs in a planar way, see~\Cref{fig:closure}(b). 
  Since $B$ has $2n+2$ buds and $2n$ legs, two buds are left unmatched. The object we obtain, which is a planar map with two dangling buds on the outer face, is the closure $\oB$ of $B$. These two buds are attached to distinct vertices as is easily checked: these are called \tdef{extremal vertices}.

\begin{figure}[!ht]
\begin{center}
\includegraphics[width=\linewidth]{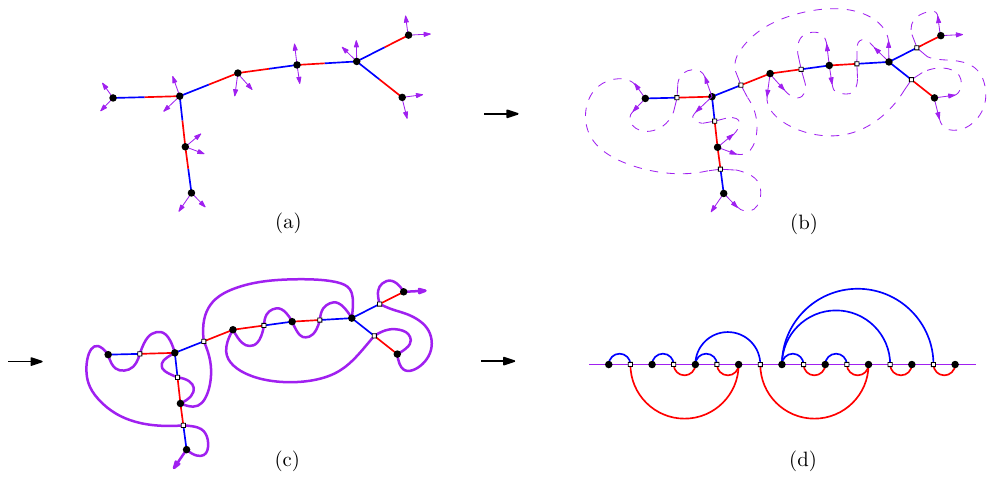}
\end{center}
\caption{(a) A bicolored blossoming tree $B$; (b) the matching of buds with legs; (c) the closure $\oB$ of $B$, where the meandric path (the concatenation of the closure-edges) is shown bolder; (d) the meandering tree $M = \bijblombil(B)$ obtained by stretching the meandric path, so as to have the blue (resp. red) arcs in the upper (resp. lower) half-plane.}
\label{fig:closure}
\end{figure}

The edges of the closure $\oB$ of a blossoming tree $B$ come in three types:
\begin{itemize}
\item \tdef{Dangling buds}: the two buds left unmatched;
\item \tdef{Tree-edges}: edges resulting from the subdivision of plain edges of $B$;
\item \tdef{Closure-edges}: edges resulting from matching buds and legs.
\end{itemize}
Every non-extremal vertex of $\oB$ is incident to exactly two closure-edges, while the two extremal vertices are incident to a unique closure-edge. We also note that extremal vertices in $\oB$ are not necessarily leaves in $B$.

\begin{lemma}\label{lem:meandering_cycle}
 For $B\in\blossoming_n$, let $\oB$ be the closure of $B$, and $\pi$ the subgraph of $\oB$ formed by the closure-edges. Then we have
 \begin{itemize}
 \item $\pi$ is a Hamiltonian path from one extremal vertex to the other;
 \item $\pi$ splits half-edges of $B$ by color;
 \item For any tree-edge $e = \{u,v\}$ of $\oB$, with $u$ a tree-vertex and $v$ an edge-vertex, let $\pi_e$ be the unique subpath of $\pi$ from $u$ to $v$, and $\sigma_e = \pi_e \cup \{e\}$, which is a cycle. Then the interior of $\sigma_e$ is on the right of $e$ traversed from $u$ to $v$.
 \end{itemize}
\end{lemma}
\begin{proof}
For the first statement, we proceed by induction on $n$. The base case $n = 1$ is easily checked. 
Let $n\geq 2$, and assume the first statement holds for size $n-1$. Let $B \in \blossoming_n$ and $\oB$ its closure. There are thus $n$ edge-vertices and $n+1$ tree-vertices in $\oB$. We then observe that, if a bud in $\oB$ is not directly succeeded by another bud of the same vertex in the \ccw-direction, then it is matched with the next leg from the next adjacent edge-vertex, with contour distance 1. We say that they form a \emph{short pair}, and clearly among the two buds of a tree-vertex, at least one is in a short pair. There are thus at least $n+1$ short pairs, while there are only $n$ edge-vertices in $\oB$, meaning that there is some edge-vertex $v_e$ of $\oB$, corresponding to some edge $e = \{u, u'\}$ in $B$, whose both legs are in short pairs with two buds $b, b'$, with $b$ on $u$ and $b'$ on $u'$. Such an edge $e$ is called a \emph{short edge} of $B$. Let $B' \in \blossoming_{n-1}$ be the tree obtained from $B$ by contracting $e$ in $B$ into a new vertex $u_0$ and removing $b, b'$; we call $B'$ the $e$-contraction of $B$. By the induction hypothesis, the closure-edges of $\oB'$ form a Hamiltonian path $\pi'$. It is clear that expanding the occurrence of $u_0$ in $\pi'$ into the path from $u$ to $v_e$ to $u'$ via the short pairs of $b$ and $b'$ gives a Hamiltonian path in $\oB$. 
 
 The second statement comes from the fact that, by construction, the two half-edges of an edge in $B$, which are of different colors, are on different sides of $\pi$, and half-edges of a vertex in $B$ are split by its buds, through which $\pi$ goes, into two groups of each color. 
 
For the third statement, we again proceed by induction on $n$. The base case $n=1$ is clear. Let $n\geq 2$, and assume the third statement holds for size $n-1$. Let $B \in \blossoming_n$, and let $e$ be a tree-edge of $\oB$, with the above notation in the statement. If $u$ and $v$ are consecutive along $\pi$, then $\sigma_e$ is a bi-gon, and thus $e$ has the interior of $\sigma_e$ on its right, since the matchings are performed in ccw order around $B$. If not, let $e'$ be a short edge of $B$, and let $B'$ be the $e'$-contraction of $B$. 
Let $\sigma_e'$ be the cycle of $\overline{B'}$ resulting from $\sigma_e$ after contraction. Since $u$ and $v$ are not consecutive along $\pi$, the tree-edge $e$ does not belong to $e'$, hence it does not collapse when contracting $e'$, and so it is on $\sigma_e'$. By induction it has the interior of $\sigma_e'$ on its right. Thus $e$ must have the interior of $\sigma_e$ on its right in $\oB$. 
\end{proof}

The Hamiltonian path $\pi$ of $\oB$ in~\Cref{lem:meandering_cycle} is called the \tdef{meandric path} of $\oB$. From the first statement of~\Cref{lem:meandering_cycle}, for $B\in\blossoming_n$, we may stretch the meandric path of $\oB$ into the horizontal segment $\{0\leq x\leq n,y=0\}$ with $2n+1$ equally-spaced vertices, along with tree-edges as semi-circular arcs. By the second statement of~\Cref{lem:meandering_cycle}, this can be done in a unique way with the blue (resp. red) half-edges of $B$ turned into the arcs above (resp. below) the segment. Let $M$ be the arc-diagram thus obtained. Then the third statement of~\Cref{lem:meandering_cycle} ensures that $M \in \mtree_n$. We define \tdef{$\bijblombil$} as the mapping that sends $B$ to $M$.

In order to prove that $\bijblombil$ is the inverse of $\bijbilblom$ we will need the following.

\begin{lemma}\label{lem:side}
Let $M\in\mtree_n$, and $S$ be any segment of length $1/2$ connecting two adjacent points on the horizontal line. Let then $\sigma_S$ be the unique embedded cycle formed by $S$ and a concatenation of arcs of $M$. Then $S$ traversed from its black end to its white end has the interior of $\sigma_S$ on its left. 
\end{lemma}
\begin{proof}
Let $G$ be the underlying graph of $M$. Note that the stated existence and uniqueness of $\sigma_S$ just follows from the fact that $G$ is a tree. 
Let $b_S$ be the black end and $w_S$ the white end of $S$. Assume $b_S$ is the left end of $S$. Let $a=(b_S',w_S)$ be the upper arc incident to $w_S$. Let $P_\bullet(S)$ be the set of black points between $b_S'$ and $b_S$ (including these two points) and let $P_\circ(S)$ be the set of white points between $b_S'$ and $b_S$ on the horizontal axis. For $w\in P_\circ(S)$ the upper arc incident to $w$ is below~$a$. By planarity it has to end at a point in $P_\bullet(S)$. And the lower arc incident to $w$ can not end on the right of $w_S$, otherwise with $a$ it would form a flawed pair, which is not possible by~\Cref{lem:flawed_meandering}. 
 
 Let $V_S$ and $E_S$ be respectively the sets of vertices and edges of $G$ corresponding to the points in $P_\bullet(S)$ and $P_\circ(S)$. Since all edges in $E_S$ connect two vertices in $V_S$, and since $|E_S|=|V_S|-1$, we conclude that $(V_S,E_S)$ form an induced subtree of~$G$. Hence, the path of $G$ between the vertices for $b_S'$ and $b_S$ passes only by vertices in $V_S$ and edges in $E_S$, hence in $M$ the corresponding sequence of arcs passes only by points in $S$, below the arc $a$. Since $\sigma_S$ is formed by $S$, $a$ and this sequence of arcs, we conclude that $S$ traversed from black end to white end has the interior of $\sigma_S$ on its left. The proof is similar if $b_S$ is the right end of $S$. 
\end{proof}
 
\begin{prop} \label{prop:gamma-bijection}
 For $n \geq 1$, the mapping $\bijbilblom$ is a bijection from $\blossoming_n$ to $\mtree_n$, with $\bijblombil$ its inverse.
\end{prop}
\begin{proof}
 We have already seen that $\bijbilblom$ is a mapping from $\mtree_n$ to $\blossoming_n$ and $\bijblombil$ is a mapping from $\blossoming_n$ to $\mtree_n$. For $B\in\blossoming_n$, we clearly have $\bijbilblom(\bijblombil(B))=B$. For $M\in\mtree_n$, let $B=\bijbilblom(M)$, with $B$ drawn in the plane as in~\Cref{fig:opening} right. Then~\Cref{lem:side} ensures that, in the matching of buds with legs to perform the closure of $B$, the matched pairs correspond to consecutive points on the horizontal axis. Thus, the meandric path of $B$ is made by the segments between consecutive points on the horizontal axis, and thus $M$ is actually the stretched closure of $B$, i.e., $M=\bijblombil(B)$. 
\end{proof}

\begin{figure}[!ht]
\begin{center}
\includegraphics[width=\linewidth]{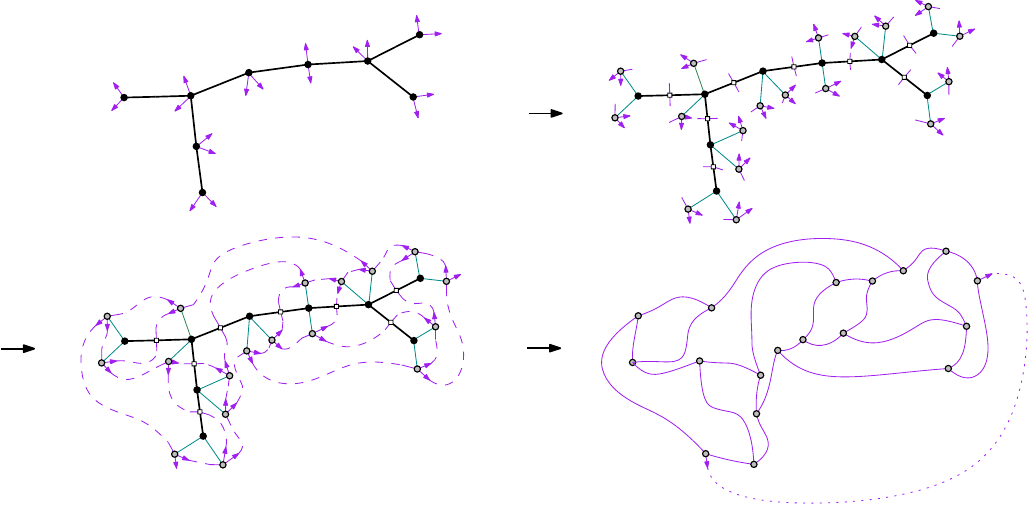}
\end{center}
\caption{The Poulalhon-Schaeffer closure bijection in the dual setting. At each leaf of the blossoming tree $B$ are attached two buds and a leg, with the leg just before the two buds in counterclockwise order around $B$; and as in our case two legs are attached at the middle of every plain edge, one on each side. Then the buds are planarly matched with legs in counterclockwise order around $B$. The output is the dual of a simple planar triangulation, with a marked edge formed by joining the two unmatched buds.}
\label{fig:closure_dual}
\end{figure}

\begin{rem} \label{rem:poulalhon-schaeffer}
As illustrated in~\Cref{fig:closure_dual}, the Poulalhon--Schaeffer closure bijection proceeds differently, by attaching two buds and a leg at each leaf of the blossoming tree. 
The two ways to perform the closure are related, for instance the unmatched buds are carried by the same nodes. 
The presentation in~\Cref{fig:closure_dual} is actually the dual formulation of the Poulalhon--Schaeffer bijection, as described by Gilles Schaeffer (personal communication).
\end{rem}

\section{The main bijection, properties and specializations} \label{sec:properties}

Combining~\Cref{prop:phi} and~\Cref{prop:gamma-bijection}, we obtain the following.

\begin{theo}\label{theo:main}
The mapping $\inttoblom := \bijbilblom \circ \treestobil$ is a bijection from $\tamint_n$ to $\blossoming_n$. Its inverse is $\blomtoint := \biltotrees \circ \bijblombil$.
\end{theo}

We now give some properties and specializations of the bijection $\inttoblom$.

\subsection{Parameter-correspondence}\label{sec:parameter}

Let $X = (T,T') \in \pair_n$, and consider its canonical drawing. For $k \in [0, n]$, the \tdef{$\nearrow$-branch} at $k$ is the left branch of $T'$ ending at $x = k$, and the number of nodes on it is given by $\dne(T')_k$. Similarly, we define the \tdef{$\swarrow$-branch} at $k$ to be the right branch of $T$ ending at $x = k$, and we denote its number of nodes by $\dsw(T)_k$, where $\dsw(T)$ is given by $\dsw(T) = \dne(\mir(T))$. We may consider $\dsw$ as the \tdef{dual degree vector} of $T$. 
or $k\in [0\isep n]$, we say that the $k$th diagonal has \tdef{bi-length} $(i,j)$ if $\dne(T')_k=i$ and $\dsw(T)_k=j$; the \tdef{bi-length vector} of $X$ is the corresponding vector of pairs of integers, see the upper part of~\Cref{fig:lengths}. 
On the other hand, for $B$ a bicolored blossoming tree, a node of $B$ is said to have \tdef{bi-degree} $(i,j)$ if it is incident to $i$ blue half-edges and to $j$ red half-edges. 
As illustrated in~\Cref{fig:lengths}, we have the following correspondence.

\begin{figure}
\begin{center}
\includegraphics[width=\linewidth]{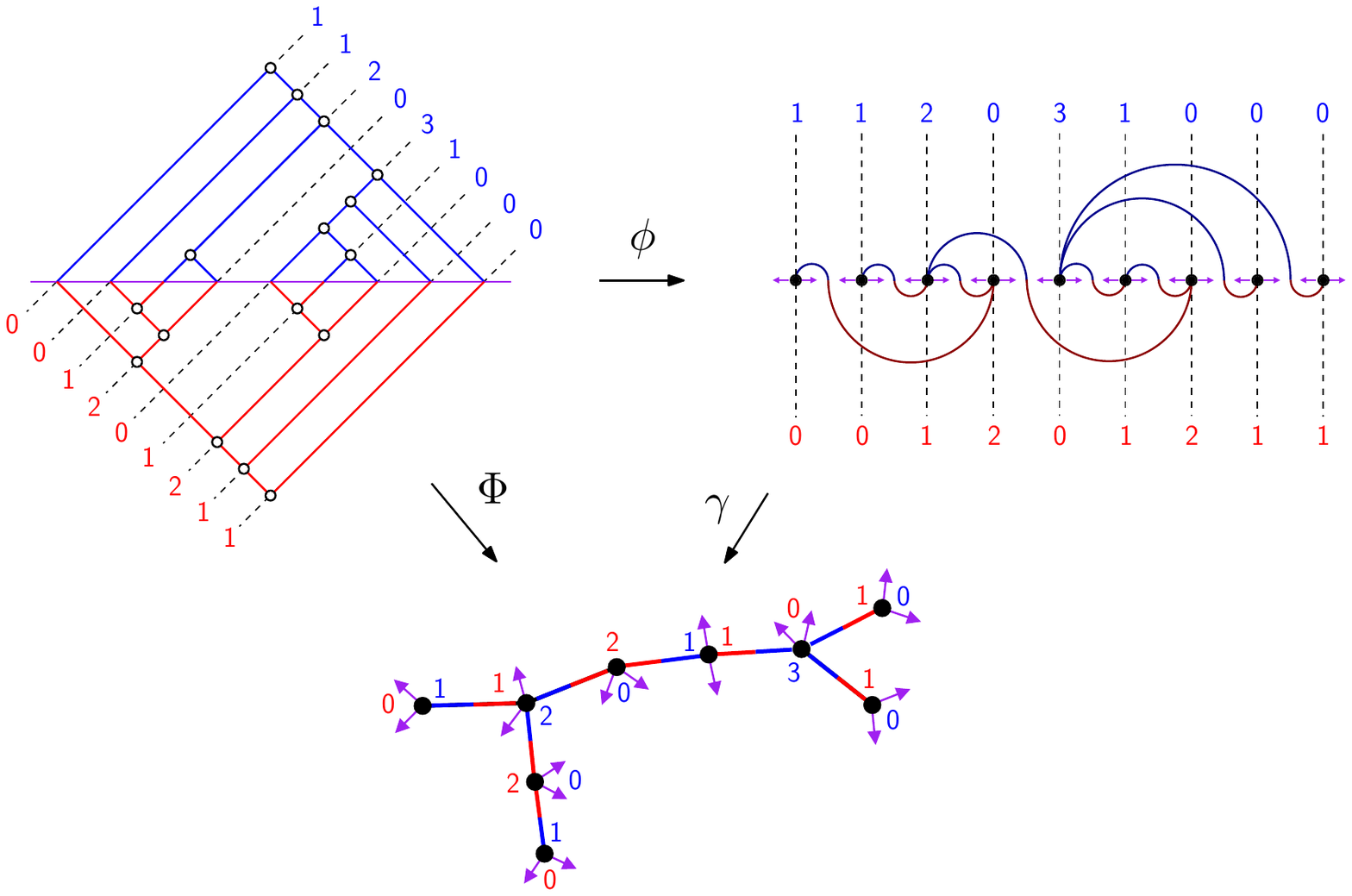}
\end{center}
\caption{Correspondence between the bi-lengths of diagonals of a Tamari interval and the bi-degrees of the nodes in the associated bicolored blossoming tree, via the bijection $\inttoblom$.}
\label{fig:lengths}
\end{figure}

\begin{lemma}\label{lem:bi-lengths}
Let $I = (T,T') \in \tamint_n$ be a Tamari interval. Then, for $i, j \geq 0$, the number of diagonals of bi-length $(i, j)$ in $I$ corresponds to the number of nodes of bi-degree $(i, j)$ in $\inttoblom(I)$.
\end{lemma}
\begin{proof}
 From~\Cref{defn:tree-diagram-drawing}, we see that the number of arcs going to the right from a black point at $x = k$ in the diagram-drawing of $T$ is $\dne(T)_k$. By~\Cref{defn:trees-to-bilabel},~\Cref{defn:bilabel-to-blossoming} and~\Cref{theo:main}, this is also the number of blue half-edges incident to the corresponding node in $\inttoblom(I)$. Similarly  $\dsw(T)_k$ is equal to the number of red half-edges incident to the same node. 
\end{proof}

We now present the needed notation for a useful corollary of~\Cref{lem:bi-lengths}. In a binary tree, we say that a leaf is of \tdef{canopy type} (or simply \tdef{type}) $1$ (resp. $0$) if it is the left child (resp. right child) of its parent. The \tdef{canopy-vector} of a binary tree $T \in \bintree_n$ is the word $\can(T) \in \{0,1\}^{n+1}$ given by the types of the leaves from left to right. We note that $\can(T)$ can be obtained from the bracket-vector $\bvec(T)$ of $T$ as the pattern of zero and non-zero entries. For $(T,T')$ a Tamari interval, the absence of flawed pair implies that the joint canopy type at every position $i \in [0, n]$ is either $\canoo,\canzz$ or $\canoz$.

\begin{rem}
 We note that our definition of canopy is slightly different from that from \cite{previlleratelleviennot17enumeration}, in which the first and the last leaves are ignored, as their types are fixed.
\end{rem}

In a blossoming tree, we say that a node is of \tdef{type} $\canoo$ (resp. $\canzz$) if it has only blue (resp. red) incident half-edges, and is of type $\canoz$ otherwise. Nodes of the first two types are called \tdef{synchronized}, and it is equivalent to their buds being side by side. 

\begin{coro}\label{coro:triple}
 For $I = (T, T')$ a Tamari interval, the number of canopy entries of type $\canoo$ (resp. $\canzz, \canoz$) in $I$ is equal to the number of nodes of type $\canoo$ (resp. $\canzz, \canoz$) in $\inttoblom(I)$. 
\end{coro}
\begin{proof}
 We apply~\Cref{lem:bi-lengths}, then observe that we obtain the type of a canopy entry of $I$ (resp. the type of a node in $\inttoblom(I)$) by replacing non-zero values in the corresponding bi-length (resp. bi-degree) by $1$ for the first component, and by replacing such values by $0$ (and $0$ by $1$) for the second component.
\end{proof}

\subsection{Commutation with duality of intervals}\label{sec:duality}

For $I = (T,T') \in \tamint_n$, by abuse of notation, we define its \tdef{mirror} $\mir(I)$ to be $(\mir(T'),\mir(T))$. It is known that $\mir(I)$ is a Tamari interval, also called the \tdef{dual} of $I$ \cite{previlleratelleviennot17enumeration}. On the other hand, for $B \in \blossoming_n$, the \tdef{dual} of $B$, denoted by $\dual(B)$, is the tree obtained from $B$ by switching the colors of half-edges. We then have the following property. 

\begin{lemma} \label{lem:duality-commutation}
 For $I \in \tamint_n$ and $B = \inttoblom(I)$, we have $\dual(B) = \inttoblom(\mir(I))$.
\end{lemma}
\begin{proof}
 We recall from~\Cref{theo:main} that $\inttoblom = \bijbilblom \circ \treestobil$. Note that the canonical (resp. smooth) drawing of $\mir(I)$ is the half-turn of the canonical (resp. smooth) drawing of $I$. 
 From this observation we get the property that $\treestobil(\mir(I))$ is $\treestobil(I)$ rotated by a half-turn. Indeed, the upper part of $\treestobil(\mir(I))$ is the diagram-drawing of $\mir(T)$, which is exactly the half-turn of the lower part of $I$; and a similar equivalence also holds for the lower part of $\treestobil(\mir(I))$. Then the mapping  $\bijbilblom$ is  such that  
 $\inttoblom(\mir(I))$ is obtained from $B=\inttoblom(I)$ by exchanging red and blue, which is exactly $\dual(B)$.
\end{proof}

\subsection{Specializations}\label{sec:specialize}

In this section, we consider several special families of Tamari intervals: synchronized, modern, modern-synchronized, infinitely modern, and Kreweras. We show that for each of these families the corresponding blossoming trees have a simple characterization.

\subsubsection{Synchronized intervals}

A Tamari interval $(T,T')$ is called \tdef{synchronized} if $\can(T)=\can(T')$. These were defined in~\cite{previlleratelleviennot17enumeration} in connection to $\nu$-Tamari-lattices (\textit{generalized Tamari lattices} therein), which are intervals of the Tamari lattice formed by elements with a fixed canopy-vector $\nu$. A synchronized interval thus corresponds to some interval in some $\nu$-Tamari lattice.

We say that a blossoming tree is \tdef{synchronized} if it has no node of type $\canoz$. Thus, at each node, the two incident buds are consecutive. We have the following as a special case of~\Cref{coro:triple}.

\begin{coro}\label{coro:sync}
 A Tamari interval $I = (T,T')$ is synchronized if and only if $\inttoblom(I)$ is synchronized. \end{coro}
\begin{proof}
 This follows directly from \ref{coro:triple} and the definition of synchronized Tamari intervals and synchronized blossoming trees.
\end{proof}

\begin{figure}
\begin{center}
\includegraphics[width=\linewidth]{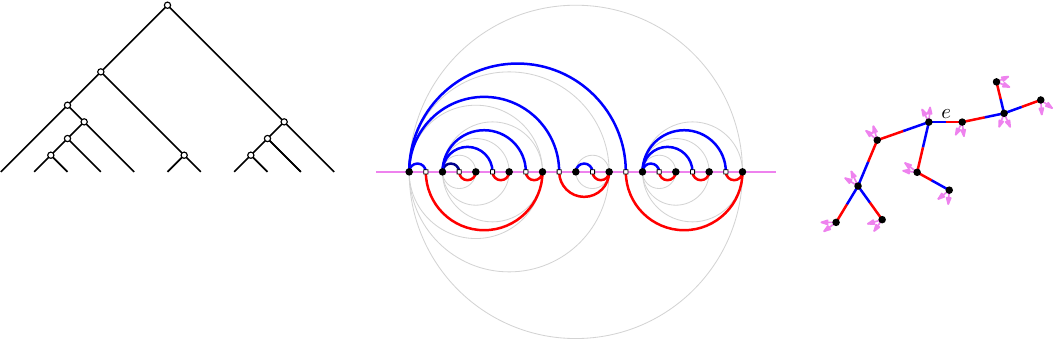}
\end{center}
\caption{A binary tree $T$, the image of $I = (T, T)$ under $\treestobil$ (meandering representation, superimposed with the smooth drawing of $I$), and the blossoming tree $\inttoblom(I)$, with $e$ the edge corresponding to the root of $T$.}
\label{fig:trivial}
\end{figure}

\begin{rem}\label{rem:trivial}
 We call Tamari intervals of the form $I = (T, T)$ \tdef{trivial intervals}, and they are synchronized. As illustrated in~\Cref{fig:trivial}, in this case, the upper representation of $\treestobil(I)$ (in the sense of~\Cref{fig:local_op_int_poset_tree}) coincides with the smooth drawing of $T$, as can be checked by induction following the binary decomposition of $T$. Moreover, for the blossoming tree $\inttoblom(I)$, let $e$ be its edge associated to the root-edge of $T$. Then, for every vertex $v\in\inttoblom(I)$, letting $e_v$
 be the incident edge at $v$ `towards' $e$, the two buds at $v$ come just after $e_v$ in \ccw-order around $v$, see~\Cref{fig:trivial} right.
\end{rem}

\subsubsection{Modern intervals}
The \tdef{rise} of a pair $I = (T,T')\in\pair_n$ is defined as $\rise(I) = ((T, \epsilon), (\epsilon, T'))$, as illustrated in~\Cref{fig:rise}. A Tamari interval $(T,T')$ is \tdef{modern} if its rise is also a Tamari interval.

\begin{figure}
\begin{center}
\includegraphics[width=12cm]{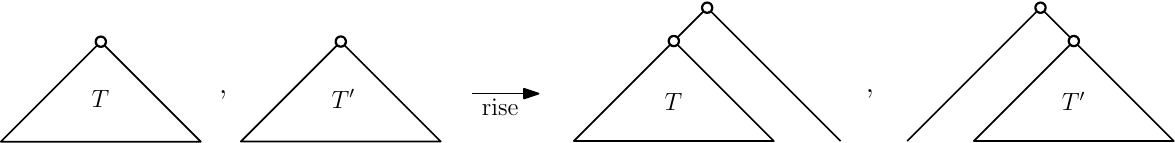}
\end{center}
\caption{The rise of a pair $(T,T') \in \pair_n$.}
\label{fig:rise}
\end{figure}

In a blossoming tree, a plain edge $e = (u, v)$ is \tdef{non-modern} if the edge following~$e$ in \cw-order around $u$ and the edge following $e$ in \cw-order around $v$ are both plain edges. A blossoming tree is called \tdef{modern} if it has no non-modern edge. In other words, a blossoming tree is modern if it avoids the ``Z'' pattern, where all three lines are plain edges, and there are no buds inside the two acute corners.

\begin{lemma}\label{lem:adjacent}
 Let $B$ be a bicolored blossoming tree, with $M = \bijblombil(B)$ the corresponding meandering tree. For $e = (u,v)$ a plain edge of $B$, in $M$ let $b_u$ and $b_v$ be the black points for $u$ and $v$ respectively, and $w$ the white point for $e$. Then, $e$ in $B$ is followed by a bud in \cw-order around $u$ (resp. $v$) if and only if $b_u$ (resp. $b_v$) is next to $w$ on the horizontal axis. 
\end{lemma}
\begin{proof}
 If the next edge after $e$ in \cw-order around $u$ is a bud, then when doing the closure of $B$, this bud is matched with the leg at $e$ on the same side. Hence, $b_u$ and $w$ are consecutive on the meandric path, thus also on the horizontal axis in $M$. Conversely, if $b_u$ and $w$ are consecutive on the horizontal axis in $\hG$, then by definition of $\gamma$ (see also~\Cref{fig:opening}), the bud at $b_u$ pointing toward $w$ corresponds to the next edge after $e$ in \cw-order around $u$ in $B$.
\end{proof}

\begin{figure}
\begin{center}
\includegraphics[width=10cm]{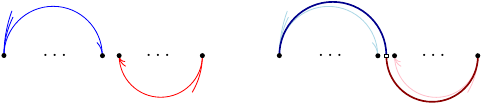}
\end{center}
\caption{For $I$ a Tamari interval, and for $M = \treestobil(I)$, the 1-gaps of $I$ correspond to the white points $w$ of $M$ whose incident arcs do not link $w$ to the two black vertices next to it on the axis.}
\label{fig:transfer_1gap}
\end{figure}

\begin{lemma}\label{lem:modern}
A Tamari interval $I = (T, T')$ is modern if and only if $\inttoblom(I)$ is modern. 
\end{lemma}
\begin{proof}
 One possible approach is via interval-posets, where a forbidden pattern for modern intervals is obtained in~\cite{rognerud18}. Here we argue via the characterization of Tamari intervals of~\Cref{lem:flawed_canonical}, that is, a pair $X = (T, T')$ is a Tamari interval if and only if the smooth drawing of $X$ has no flawed pair.
 
 We observe that the smooth drawing of $(\epsilon, T')$ is obtained from that of $T'$ by adding a point to the left and linking it to the rightmost point by an arc. Hence, the smooth drawing of $\rise(I)$ is obtained from that of $I$ by shifting the upper part by one unit to the right and adding an arc linking the leftmost and the rightmost points in both the upper and the lower part.

 Now, $I$ is modern if and only if $\rise(I)$ has no flawed pair in its smooth drawing. As $I$ is a Tamari interval, its smooth drawing has no flawed pair. Thus, the smooth drawing of $\rise(I)$ has no flawed pair if and only if the smooth drawing of $I$ has no ``1-separated pair", i.e., a pair $a, a'$ of lower and upper arc such that the right end of $a'$ is one unit to the left of the left end of $a$. In this case, the offending unit segment is called a \emph{1-gap} of $I$. Indeed, such a pair becomes a flawed pair in $\rise(I)$, with $x_r'=x_\ell$, and any other flawed pair in $\rise(I)$ would have to come from a flawed pair in $I$.
 
 Then, as shown in~\Cref{fig:transfer_1gap}, the 1-gaps of $I$ correspond exactly to the edges $e$ of the underlying tree of $M = \treestobil(I)$ such that none of the two black extremities of $e$ is next to the white point of $e$. By~\Cref{lem:adjacent}, these are exactly the non-modern edges of $\inttoblom(I)$. The absence of 1-gap in $I$, which is equivalent to $I$ being modern, is thus equivalent to $\inttoblom(I)$ being modern.
\end{proof}

\begin{rem} \label{rem:new-intervals}
 It is shown in~\cite{rognerud18} that $\rise$ is a bijection between modern intervals in $\tamint_n$ and \emph{new intervals} in $\tamint_{n+1}$ introduced and counted by Chapoton in~\cite{chapoton06}.
\end{rem}

\subsubsection{Infinitely modern intervals}

Following~\cite{rognerud18}, an interval $I = (T,T') \in \tamint_n$ is called \tdef{infinitely modern} if $\rise^k(I)$ is a Tamari interval of size $n+k$ for all $k \geq 0$.

In a blossoming tree, a simple path $\pi = v_0, \ldots, v_k$ is called \tdef{non-modern} if the edge following $(v_0,v_1)$ in \cw-order around $v_0$ and the edge following $(v_{k-1},v_k)$ in \cw-order around $v_k$ are both plain edges. A blossoming tree is called \tdef{infinitely modern} if it has no non-modern path. In other words, a blossoming tree is infinitely modern if it avoids the long ``Z'' pattern, where the diagonal is a path of plain edges, the two horizontal lines are plain edges, and there are no buds inside the two acute corners.

\begin{figure}
\begin{center}
\includegraphics[width=\textwidth]{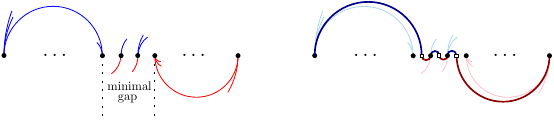}
\end{center}
\caption{Correspondence between minimal gaps of $I$ and the minimal non-modern paths of $B = \inttoblom(I)$.}
\label{fig:transfer_gap}
\end{figure}

\begin{lemma} \label{lem:infinitely-modern}
 A Tamari interval $I = (T,T')$ is infinitely modern if and only if $\inttoblom(I)$ is infinitely modern. 
\end{lemma}
\begin{proof}
 Again, one may argue either via interval-posets (where a forbidden pattern for being associated to an infinitely modern interval is obtained in~\cite{rognerud18}), or, as done here, via the characterization of Tamari intervals of~\Cref{lem:flawed_canonical}.

 A \emph{separated pair} of $I$ is a pair $(a, a')$ of arcs in the smooth drawing of $I$, where $a$ is an upper arc and $a'$ a lower one, such that $a'$ is totally on the left of $a$. The segment between $a$ and $a'$ is called a \emph{gap} of $I$. By repeating the argument in the proof of~\Cref{lem:modern} $k$ times, a separated pair of $I$ with a gap of length $k$ yields a flawed pair of $\rise^k(I)$. Thus, as $I$ is a Tamari interval, $I$ is infinitely modern if and only if $I$ has no separated pair. We say that a gap is \emph{minimal} if it contains no other gap.
 
 On the other hand, for $B$ a bicolored blossoming tree, a non-modern path $\pi = v_0, \ldots, v_k$ is called \emph{minimal} if, for $i \in [k-1]$, the next edge after $(v_i,v_{i+1})$ (resp. after $(v_i,v_{i-1})$) in \cw-order around $v_i$ is a bud. Clearly, if $B$ has a non-modern path, then it has a minimal one, for instance a shortest one.
 
 Then, as shown in~\Cref{fig:transfer_gap}, and using~\Cref{lem:adjacent}, for $k\geq 1$ the minimal gaps of length~$k$ in $I$ are in 1-to-1 correspondence with the minimal non-modern paths of length $k$ in $\inttoblom(I)$. The absence of separated pairs in $I$ is thus equivalent to the absence of non-modern paths in $\inttoblom(I)$.
\end{proof}

\begin{rem}
For $k\geq 1$, a Tamari interval $I$ is called \emph{$k$-modern} if  $\mathrm{rise}^i(I)$ is a Tamari interval for all $i\in [k]$. 
The proof of~\Cref{lem:infinitely-modern} ensures that $I$ is $k$-modern if and only if $\Phi(I)$ has no non-modern path of length at most $k$. 
\end{rem}

\subsubsection{Kreweras intervals}
\label{subsub:kreweras}
The \tdef{Kreweras lattice} of size $n$ is the set $\noncross_n$ of non-crossing partitions of size $n$ endowed with the refinement order~\cite{Kre72}; the top (resp. bottom) element being the partition with a single block (resp. with $n$ blocks). There is a standard bijection $\iota$ between $\bintree_n$ and $\noncross_n$: for $T \in \bintree_n$, with nodes labeled by infix order, the associated non-crossing partition $\iota(T)$ is the partition of the nodes into right branches. 
\begin{rem}
If we consider the non-crossing partition given by the partition of the nodes of $T$ into left branches, then we obtain the Kreweras complement $\overline{\pi}$ of $\pi=\iota(T)$, see Figure~\ref{fig:complement_bij}.
\end{rem}

\begin{figure}
\begin{center}
\includegraphics[width=8cm]{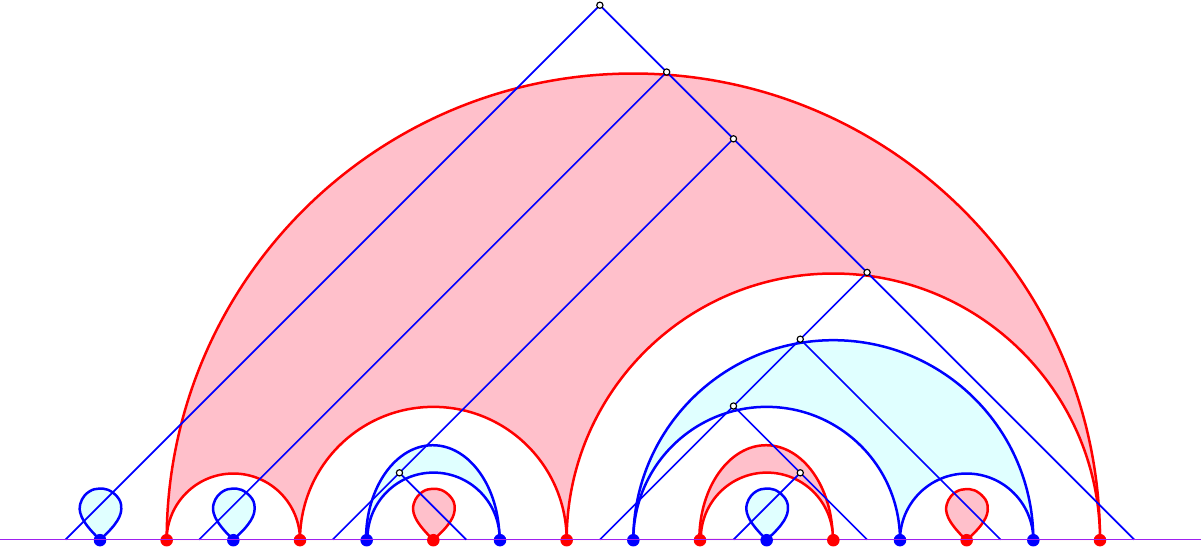}
\end{center}
\caption{A binary tree $T$ given by its canonical drawing, the associated non-crossing partition $\pi=\iota(T)$ in red, and its Kreweras complement $\overline{\pi}$ in blue, which also corresponds to the 
partition of the nodes into left branches.} 
\label{fig:complement_bij}
\end{figure}

By a slight abuse of notation, the pairs $(T,T')\in\pair_n$ that are mapped to Kreweras intervals by the mapping $\iota$ are called \tdef{Kreweras intervals}.

In a meandering diagram $M\in\mdiag_n$, a \tdef{non-Kreweras pair} is a pair made of a lower arc $a_\downarrow$ and an upper arc $a_\uparrow$ such that, for $x_{\ell}$ and $x_r$ (resp. $x_{\ell}'$ and $x_r'$) the abscissas of the left and the right end of the lower (resp. upper) arc, we have $x_{\ell}<x_{\ell}'< x_r<x_r'$, see Figure~\ref{fig:transfer_kreweras} left.

\begin{figure}
\begin{center}
\includegraphics[width=8cm]{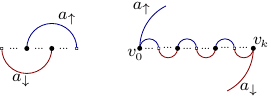}
\end{center}
\caption{(Left) A non-Kreweras pair of arcs in a meandering diagram. (Right) A minimal non-Kreweras path of a blossoming tree, in the meandering representation.}
\label{fig:transfer_kreweras}
\end{figure}

\begin{figure}
\begin{center}
\includegraphics[width=0.48\linewidth]{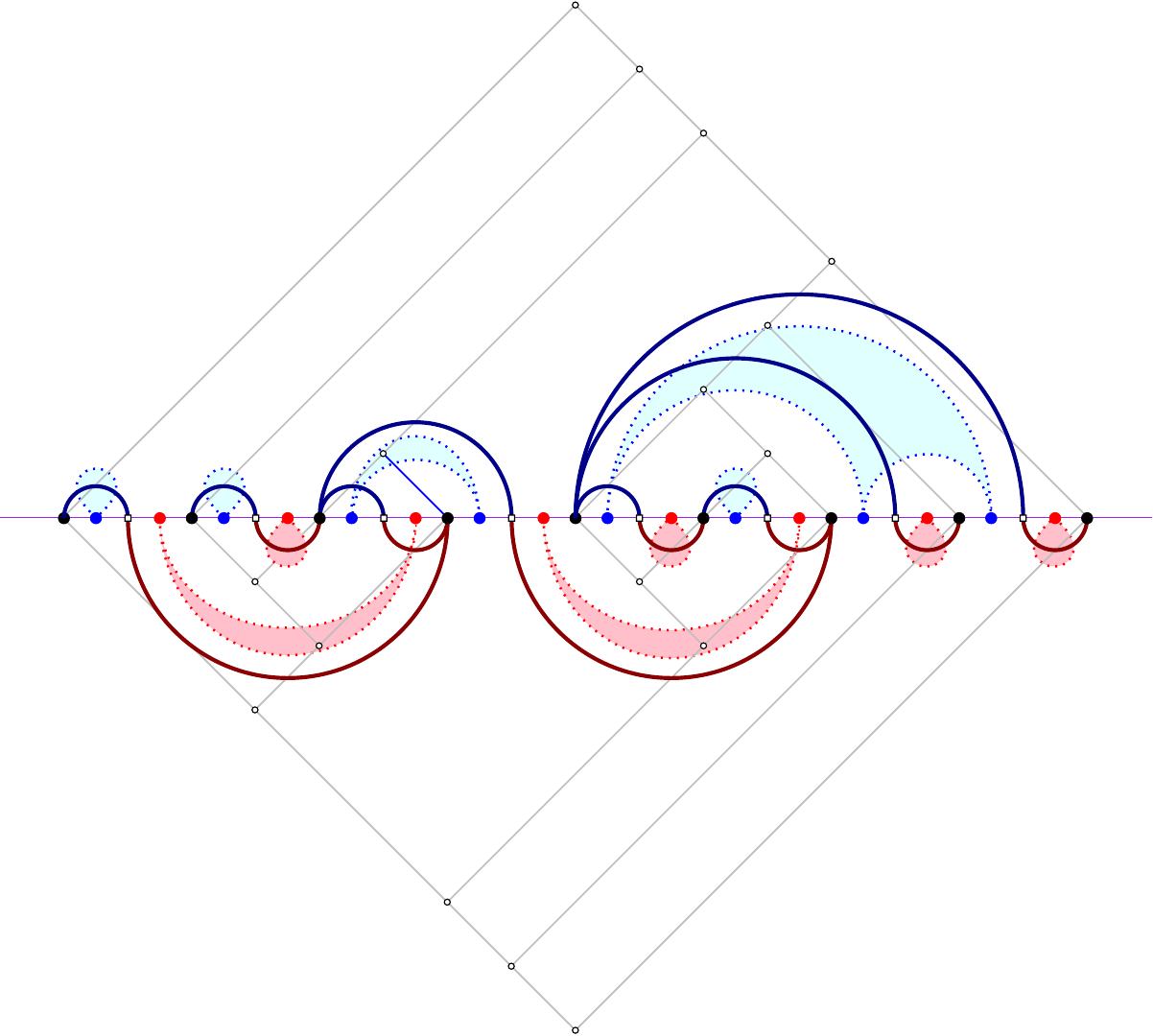}\hspace{.2cm}\includegraphics[width=0.48\linewidth]{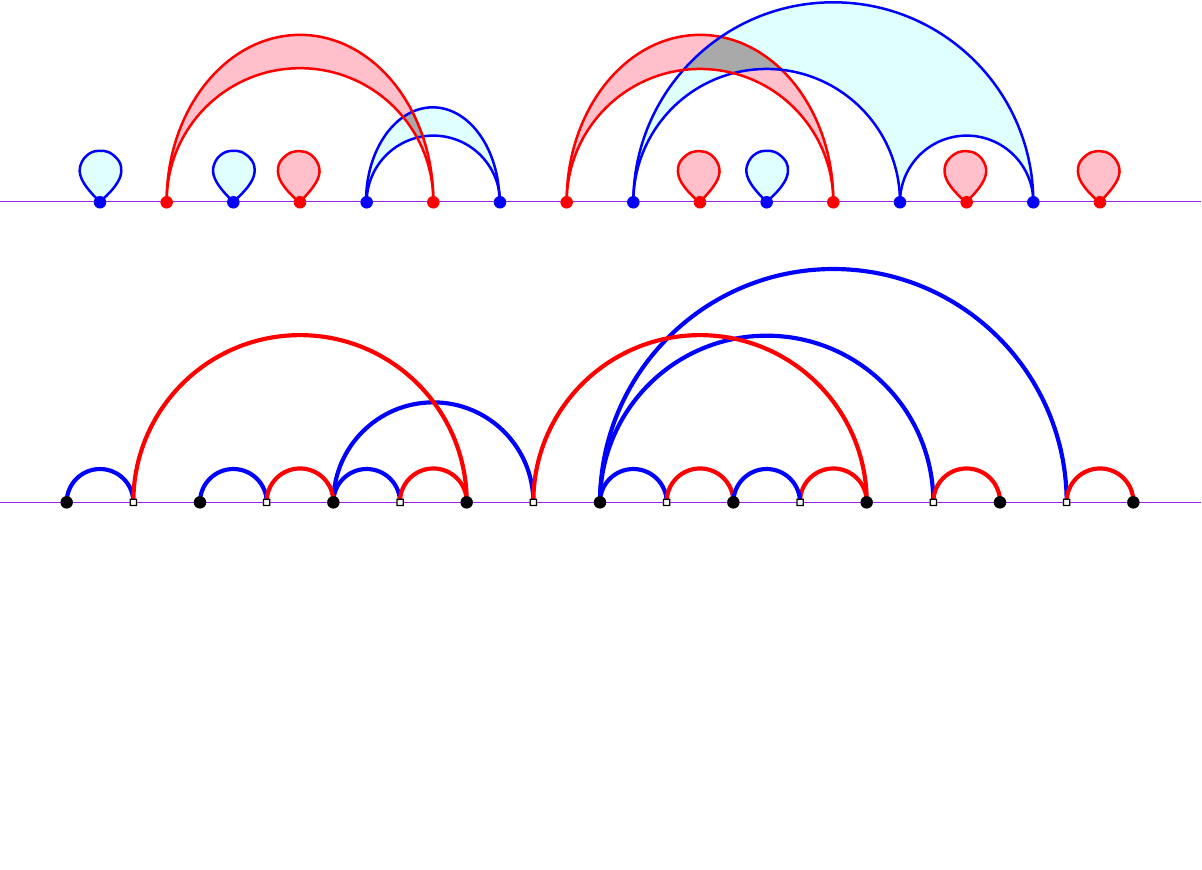}
\end{center}
\caption{Left: for the pair $X=(T,T')$ in the top row of~\Cref{fig:examples_phi}, with $\pi=\iota(T)$ and $\pi'=\iota(T')$, the drawing shows $M=\phi(X)$ superimposed with $\pi$ in the lower part, and with $\overline{\pi'}$ in the upper part. Right: the joint representation $J(\pi,\pi')$ above, and the upper representation $M^\uparrow$ below. The pair $(\pi,\pi')$ is not a Kreweras interval, as ensured by the presence of a crossing in $J(\pi,\pi')$, which is equivalent to the presence of a crossing in $M^\uparrow$.} 
\label{fig:kreweras_represent}
\end{figure}

\begin{figure}
\begin{center}
\includegraphics[width=0.55\linewidth]{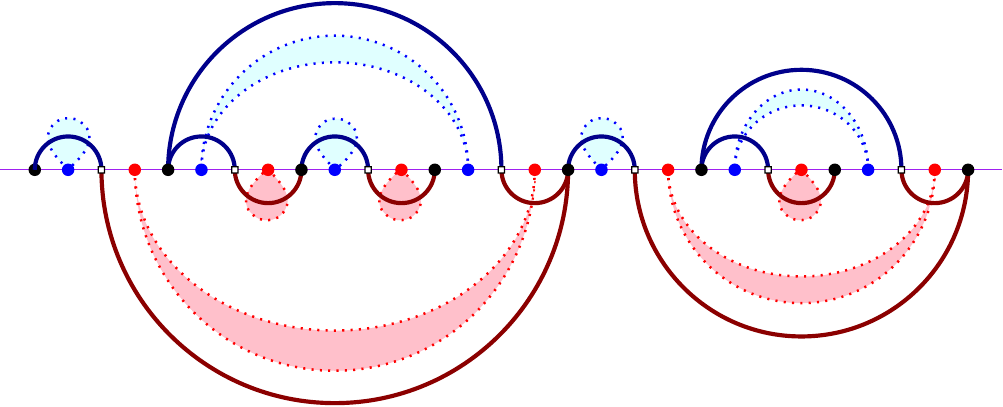}\hspace{.4cm}\includegraphics[width=0.41\linewidth]{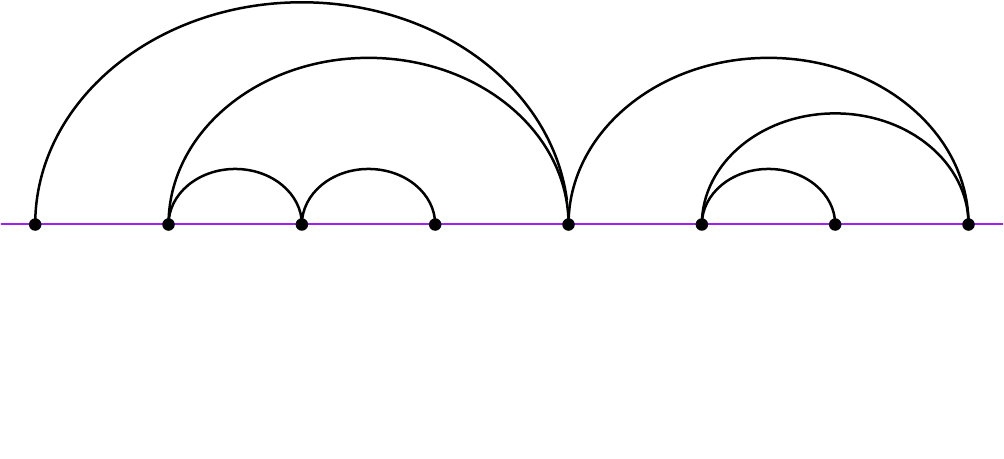}
\end{center}
\caption{Consider the Kreweras interval formed by the non-crossing partitions $\pi=\{\{1,4\},\{2\},\{3\},\{5,7\},\{6\}\}$ and $\pi'=\{\{1,4,5,7\},\{2,3\},\{6\}\}$. The left drawing shows $\pi$ in the lower part and the Kreweras complement
$\overline{\pi'}=\{\{1\},\{2,4\},\{3\},\{5\},\{6,7\}\}$ in the upper part, superimposed with the meandering tree $M$ associated to $(\pi,\pi')$. The right drawing shows the non-crossing tree induced by $M$.}
\label{fig:kreweras_noncrossing}
\end{figure}

\begin{lemma}\label{lem:meandering_krew}
A pair $X=(T,T')\in\pair_n$ is a Kreweras interval if and only if $M=\phi(X)$ is a meandering tree without non-Kreweras pairs.
\end{lemma}
\begin{proof}
Let $\pi=\iota(T)$ and $\pi'=\iota(T')$. Let $J(\pi,\pi')$ be the superimposition of $\pi$ and $\overline{\pi'}$, with the parts of $\pi$ red and the parts of $\overline{\pi'}$ blue, and the points alternatively blue and red on the line, as in~\Cref{fig:complement_bij} (which has $\pi'=\pi$). It is easy to see that $J(\pi,\pi')$ is non-crossing if and only if $(\pi,\pi')$ is an interval in the Kreweras lattice. Indeed by construction of the Kreweras complement, $\pi'$ is the partition of the red points according to the connected components of the upper half-plane cut by the blue regions. Hence, the partition $\pi$ on red points has to satisfy $\pi\leq \pi'$ for $J(\pi,\pi')$ to be non-crossing.

On the other hand, the \emph{upper representation} of $M$ is the arc-diagram $M^\uparrow$ obtained by flipping the lower arcs of $M$ upwards. Clearly $M^\uparrow$ is crossing-free if and only if $M$ has no flawed pairs nor non-Kreweras pairs. By~\Cref{lem:flawed_meandering} this is equivalent to $M$ being a meandering tree without non-Kreweras pairs. 

Now, as illustrated in~\Cref{fig:kreweras_represent}, there is a simple link between $J(\pi,\pi')$ and $M^\uparrow$. For each point $p$ of $\pi$ (resp. $\overline{\pi'}$), let the \emph{attached point} of $p$ be the rightmost (resp. leftmost) point of the block to which $p$ belongs. Then $p$ yields an arc in $M^\uparrow$ connecting the white point just on the left (resp. right) of $p$ to the black point just on the right (resp. left) of its attached point. From this correspondence it is easily checked that $J(\pi,\pi')$ is crossing-free if and only if $M^\uparrow$ is crossing-free.
\end{proof}

\begin{rem}\label{rk:non-crossing-tree}
Lemma~\ref{lem:meandering_krew} ensures that Kreweras intervals form a subfamily of Tamari intervals, which is well-known, see~\cite{bernardi09} and references therein. These Tamari intervals are called \emph{exceptional} in~\cite{rognerud18}. Note that the meandering trees $M$ whose upper representation is crossing-free correspond to non-crossing trees for the operation of~\Cref{fig:local_op_int_poset_tree} performed from right to left. An example is shown in Figure~\ref{fig:kreweras_noncrossing}. This recovers the fact that the interval-poset trees for exceptional Tamari intervals are the non-crossing trees~\cite{rognerud18}. 
\end{rem}

In a blossoming tree, a simple path $\pi = v_0, \ldots, v_k$ with $k \geq 1$ is called \tdef{non-Kreweras} if the edge following $(v_0, v_1)$ in \ccw-order around $v_0$ is a plain edge, and the same holds for $(v_{k-1}, v_k)$ with $v_k$. A blossoming tree is called \tdef{Kreweras} if it has no non-Kreweras path. 

\begin{lemma}
 A Tamari interval $I = (T, T')$ is Kreweras if and only if $\inttoblom(I)$ is Kreweras. 
\end{lemma}
\begin{proof}
By~\Cref{lem:meandering_krew} it suffices to show that $M\in\mtree_n$ has a non-Kreweras pair if and only if $B=\bijbilblom(M)$ has a non-Kreweras path. 

Assume $M$ has a non-Kreweras pair of arcs $a_\uparrow, a_\downarrow$. Let $v$ and $w$ be respectively the left endpoint of $a_\uparrow$ and the right endpoint of $a_\downarrow$. Let $V'$ be the set of vertices of $B$ corresponding to black points from $v$ to $w$ on the horizontal line (both ends included), and let $E'$ be the set of plain edges of $B$ whose white point is between $v$ and $w$ on the horizontal line. Note that $|E'| = |V'| - 1$. Moreover, by planarity, as limited by $a_\uparrow$ and $a_\downarrow$, every edge in $E'$ must connect two vertices in $V'$. Hence, as $B$ is a tree, the subgraph $H = (V',E')$ is a subtree of $B$, which implies that there is a path $\pi$ in $H$ from $v$ to $w$. In $M$, the path $\pi$ stays between $v$ and $w$, hence is non-Kreweras due to the two edges containing $a_\uparrow$ and $a_\downarrow$. 

Conversely, suppose that $B$ has a non-Kreweras path. Then it has a non-Kreweras path $\pi = v_0, \ldots, v_k$ that is \emph{minimal}, i.e., for $i \in [k-1]$, the next edge after $(v_i,v_{i+1})$ (resp. after $(v_i,v_{i-1})$) in \ccw-order around $v_i$ is a bud. The situation in $M$ is as shown in the right-part of~\Cref{fig:transfer_kreweras} (up to exchanging $v_0$ and $v_k$). Let $a_{\uparrow}$ be the upper arc of the next edge after $(v_0,v_1)$ in \ccw-order around $v_0$, and $a_{\downarrow}$ the lower arc of the next edge after $(v_{k-1},v_k)$ in \ccw-order around $v_k$. Then, by planarity and the absence of flawed pair in $M$, we know that $a_{\uparrow}$ ends on the right of $v_k$, and similarly $a_\downarrow$ ends on the left of $v_0$. Thus, $a_\uparrow$ and $a_\downarrow$ form a non-Kreweras pair of arcs in $M$.
\end{proof}

\begin{figure}[!ht]
 \centering
 \includegraphics[page=1,width=0.6\textwidth]{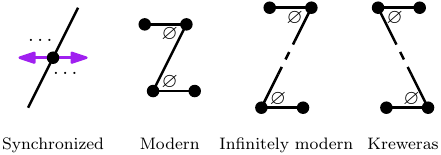}
 \caption{Forbidden patterns of blossoming trees for subfamilies of Tamari intervals.}
 \label{fig:forbidden-pattern}
\end{figure}

 \begin{rem}\label{rem:inf_modern_mirror}
 The forbidden patterns for the considered subfamilies are illustrated in Figure~\ref{fig:forbidden-pattern}. 
We note that a blossoming tree is Kreweras if and only if its reflection is infinitely modern.
 \end{rem}

\section{Counting results}
\label{sec:counting}

\subsection{Tamari intervals}
From~\Cref{theo:main} we recover the counting formula for Tamari intervals as follows.

\begin{prop}\label{prop:count_all_intervals}
 For $n \geq 1$, the number of Tamari intervals of size $n$ is
 \[
 \frac{2}{n(n+1)}\binom{4n+1}{n-1}.
 \]
\end{prop}
\begin{proof}
 By~\Cref{theo:main}, we have $|\tamint_n|=|\blossoming_n|$. The formula for $|\blossoming_n|$ was obtained in~\cite{PS06} (in the form of counting \emph{balanced} blossoming trees) using a contour encoding. For completeness, and in view of recovering a bivariate refinement due to Bostan, Chyzak and Pilaud \cite{bostan23}, we give here a slightly different encoding.
 
 Let $\blossomedge_n$ be the set of bicolored blossoming trees of size $n$ with a marked edge. For $T \in \blossomedge_n$ with $e$ the marked edge, let $v_0, \ldots, v_n$ be the vertices of $T$ ordered by first visit in a \ccw-tour around $T$, starting at the middle of $e$ along its red half-edge. Note that $T$ can be viewed as a pair of rooted trees, each rooted at a node adjacent to $e$. For a node $v_i$, its children are split by the two incident buds into 3 groups: left, middle and right. The respective sizes of these 3 groups are denoted by $\ell_i, m_i, r_i$.
 
 Let $\cS_n$ be the set of sequences $(a_0, \ldots, a_{3n+2})$ with all $a_i \in \mathbb{N}$ such that $\sum_{i=0}^{3n+2} a_i = n - 1$, and $\cvS_n \subset \cS_n$ the set of such sequences satisfying $\sum_{j=0}^{3i+2} a_j \geq i$ for all $0 \leq i \leq n-1$. By considering the sequence $(b_0, \ldots, b_{n})$ with $b_i = a_{3i} + a_{3i+1} + a_{3i+2}$ and applying the cycle lemma (see \cite{cycle-lemma}) to it, we see that $\cvS_n$ is in 2-to-$(n+1)$ correspondence with $\cS_n$. Let $\vVec(T) = (\ell_0, m_0, r_0, \ldots, \ell_n, m_n, r_n)$. It is clear that $\vVec(T) \in \cvS_n$, as $\sum_{j=0}^{i} (\ell_j + m_j + r_j)$ is the number of children of already visited nodes up to $v_i$, which is at least $i$ for the first rooted tree and at least $i - 1$ for the second rooted tree, as the two roots are the only nodes not counted here. For any sequence from $\cvS_n$, we may construct the first rooted tree using the sequence up to the first index $i$ such that $\sum_{j=0}^{i} (\ell_j + m_j + r_j) = i$, and the second rooted tree with the rest of the sequence. Therefore, the mapping $\vVec$ is a bijection from $\blossomedge_n$ to $\cvS_n$. We thus have
\[
 |\blossoming_n| = \frac{1}{n}|\blossomedge_n| = \frac1{n}|\cvS_n| = \frac{2}{n(n+1)}|\cS_n| = \frac{2}{n(n+1)}\binom{4n+1}{n-1}. \qedhere
\]
\end{proof}

\subsection{Trivariate series according to canopy-parameters} \label{sec:trivariate}

From~\Cref{coro:triple} we derive here formulas for the trivariate counting series of Tamari intervals according to the 3 types of canopy entries. 

\begin{prop}\label{prop:trivariate}
 Let $I_{i,j,m}$ be the number of Tamari intervals of size $n = i + j + m - 1$ with $i, j, m$ the number of canopy-entries of type $\canoo, \canzz, \canoz$ respectively. Let $A \equiv A(x, y, z)$ and $B \equiv B(x, y, z)$ be the trivariate series defined by
 \begin{equation} \label{eq:defn-canopy-series}
 A = \frac{1}{(1 - B)^2} \left(y + \frac{zA}{1-A} \right), \quad B = \frac{1}{(1 - A)^2} \left(x + \frac{zB}{1-B} \right).
 \end{equation}
 Then we have
 \begin{equation}\label{eq:first_exp_canopy}
 \sum_{i,j,m} (i + j + m - 1) I_{i,j,m} x^iy^jz^m = A B, 
 \end{equation}
 \begin{equation}\label{eq:second_exp_canopy}
 F(x, y, z) \equiv \sum_{i,j,m} I_{i,j,m} x^iy^jz^m = \frac{xA}{1 - A} + \frac{yB}{1 - B} + \frac{zAB}{(1 - A) (1 - B)} - AB.
 \end{equation}
\end{prop}
\begin{proof}
 By~\Cref{sec:parameter}, $I_{i,j,m}$ is the number of bicolored blossoming trees having $i$ nodes of type $\canoo$, $j$ nodes of type $\canzz$, and $m$ nodes of type $\canoz$. A \tdef{planted blossoming tree} $T$ is one of the two trees obtained after cutting a plain edge in a bicolored blossoming tree, with the dangling half-edge of the cut-out edge as the \tdef{root}. It is red-planted (resp. blue-planted) if the root is red (resp. blue). Let $A \equiv A(x,y,z)$ and $B \equiv B(x,y,z)$ be respectively the counting series of red-planted and blue-planted blossoming trees, with $x,y,z$ marking the number of nodes of types $\canoo, \canzz, \canoz$ respectively, with the types defined similarly as in bicolored blossoming trees.

 We may decompose the two types of planted blossoming trees at the root, separated by the two buds into three sequences of sub-trees, two with planted blossoming trees planted at opposing color, one with those planted at the same color, whose emptiness determines the type of the node. See~\Cref{fig:AB_decomposition} for an illustration for $A$, that for $B$ is similar. Thus $A$ and $B$ satisfy~\eqref{eq:defn-canopy-series}.
 
 The trivariate counting series of bicolored blossoming trees with a marked plain edge, denoted by $F^-$, is clearly equal to $AB$, which gives~\eqref{eq:first_exp_canopy}. For the second formula, we note that the counting series of bicolored blossoming trees with a marked node, denoted by $F^\bullet$, is
 \[
 F^\bullet = \frac{xA}{1 - A} + \frac{yB}{1 - B} + \frac{zAB}{(1 - A) (1 - B)},
 \]
 where each term corresponds to the marked node being of type $\canoo, \canzz, \canoz$ respectively.  We have $F = F^\bullet - F^-$, as there are one more nodes than plain edges, and as bicolored blossoming trees have no symmetry. We thus get~\eqref{eq:second_exp_canopy}.
\end{proof}

\begin{figure}
\begin{center}
\includegraphics[width=0.8\textwidth]{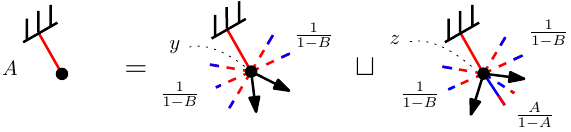}
\caption{Combinatorial decomposition of red-planted blossoming trees, indicating the contributions to the counting series.\label{fig:AB_decomposition}}
\end{center}
\end{figure}

\begin{rem}
 The series $R = A / (1 - A)$ and $G = B / (1 - B)$ are specified by
 \[
 R = (1 + R) (1 + G)^2 (y + zR), \quad G = (1 + G) (1 + R)^2 (x + zG),
 \]
 and~\eqref{eq:second_exp_canopy} becomes
 \[
 F(x, y, z) = xR + yG + zRG - \frac{RG}{(1 + R) (1 + G)}.
 \]
 This is exactly the expression obtained in~\cite{fusy19} via the Bernardi--Bonichon bijection composed with a bijection from minimal Schnyder woods to a certain subclass of unrooted binary trees. The derivation in~\cite{fusy19} is however less direct, as the system obtained there involves a third series ($B$ therein), which can be eliminated by algebraic manipulations.
\end{rem}

\subsection{A formula by Bostan, Chyzak and Pilaud} \label{sec:biv_formula}

In a recent work~\cite{bostan23}, Bostan, Chyzak and Pilaud are interested in the enumeration of Tamari intervals with respect to certain statistics, motivated by geometric objects called diagonals of the associahedra. One of their main results concerns the number $J_k(n)$ of intervals $(T, T') \in \tamint_n$ such that $\can_i(T) = \can_i(T')$ for $k+2$ values of $i$. Their theorem was stated in terms of covering relations; it is clearly equivalent to the following statement.

\begin{prop}[{\cite[Theorem~1]{bostan23}}] \label{prop:bostan}
For any $n \geq 1, k \geq 0$, we have
\begin{equation} \label{eq:refinedtamari}
 J_k(n) = \frac{2}{n(n+1)} \binom{3n}{k} \binom{n+1}{k+2}.
\end{equation}
\end{prop}
\begin{proof}
 In the notation of~\Cref{sec:trivariate}, $J_k(n)$ is given by $\sum_{i+j=k+2} I_{i,j,n-1-k}$. By setting $x = y$, we note that $A(x, x, z) = B(x, x, z)$ by symmetry for the series $A, B$ in~\Cref{prop:trivariate}. Setting moreover $x = xt$ and $z = t$ in~\eqref{eq:first_exp_canopy}, we have
 \[
 \sum_{k,n} n J_k(n) x^{k+2} t^{n+1} = A^2, \text{ where } A = \frac{t}{(1 - A)^2} \left( x + \frac{A}{1 - A} \right).
 \]
 We thus obtain~\eqref{eq:refinedtamari} using the Lagrange inversion formula:
 \begin{align*}
 [x^{k+2}t^{n+1}]A^2 &= [x^{k+2}] \frac{2}{n+1} [y^{n-1}] \left(\frac{x}{(1-y)^2} + \frac{y}{(1-y)^3} \right)^{n+1}\\
 &= \frac{2}{n+1} [y^{k}] \binom{n+1}{k+2} \frac1{(1-y)^{3n-1-k}} = \frac{2}{n+1} \binom{n+1}{k+2} \binom{3n}{k}.
 \end{align*}

 One can also proceed bijectively by a simple adaptation of the proof of~\Cref{prop:count_all_intervals}. With the notation therein, $nJ_{k-2}(n)$ is equal to the cardinality of the subset $\blossomedge_{n,k}$ of $\blossomedge_n$ of blossoming trees where exactly $k$ of the $(n + 1)$ nodes $v_i$ satisfy $m_i = 0$, meaning that the two buds at $v_i$ are consecutive. We call every entry of the form $a_{3k+1}$ in $(a_0, \ldots, a_{3n+2}) \in \cS_n$ a \tdef{mid-entry}. Let $\cS_{n,k}$ (resp. $\cvS_{n,k}$) be the subset of $\cS_n$ (resp. $\cvS_n$) of sequences with exactly $k$ mid-entries equal to $0$. The mapping $\vVec$ is a bijection from $\blossomedge_{n,k}$ to $\cvS_{n,k}$, and the cyclic lemma ensures that $\cvS_{n,k}$ is in 2-to-$(n+1)$ correspondence with $\cS_{n,k}$. Hence,
 \[
 J_{k-2}(n) = \frac1{n} |\blossomedge_{n,k}| = \frac1{n} |\cvS_{n,k}| = \frac{2}{n(n+1)} |\cS_{n,k}| = \frac{2}{n(n+1)} \binom{n+1}{k} \binom{3n}{k-2},
 \] 
 the factor $\binom{n+1}{k}$ accounting for the choice of positions of mid-entries equal to $0$.
\end{proof}

We note that the original proof in~\cite{bostan23} is more involved and requires solving a certain functional equation with the help of a computer.

\subsection{Synchronized intervals}

We recall that, in a synchronized blossoming tree, there is no node of type $\canoz$, thus for each node, its two buds are consecutive. It also means that the half-edges adjacent to each node are monochromatic. We define the \tdef{reduction} of a synchronized blossoming tree to be the tree obtained by merging the two buds at each node into one, and coloring red (resp. blue) the nodes incident to red half-edges (resp. blue half-edges) only. We thus obtain a so-called \tdef{bicolored 1-blossoming tree}, i.e., an unrooted plane tree with one bud per node, and whose nodes are colored red or blue such that adjacent nodes have different colors. As usual, the size of such a tree is its number of nodes minus 1. \Cref{coro:sync} then ensures that synchronized Tamari intervals of size $n$ are in bijection with bicolored 1-blossoming trees of size $n$, and~\Cref{coro:triple} implies that the number of $1$'s (resp. $0$'s) in the common canopy word of the interval is the number of blue nodes (resp. red nodes) in the corresponding bicolored 1-blossoming tree. We then recover the following formulas for the enumeration of synchronized intervals. 

\begin{prop}
 For $n \geq 1$, let $S_n$ be the number of synchronized intervals of size $n$. For $i,j\geq 1$, let $S_{i,j}$ be the number of synchronized intervals with $i$ (resp. $j$) canopy entries of type $\canoo$ (resp. $\canzz$). Then
 \[
 S_n = \frac{2}{n (n + 1)} \binom{3n}{n - 1},\quad S_{i,j} = \frac{1}{ij} \binom{2i + j - 2}{j - 1} \binom{2j + i - 2}{i - 1}.
 \]
\end{prop}
\begin{proof}
 The formula for $S_n$ is the special case $k = n-1$ in~\eqref{eq:refinedtamari}. For the bivariate formula, by the above reduction, $S_{i,j}$ is the number of bicolored 1-blossoming tree with $i$ blue nodes and $j$ red nodes. 
 
 A \tdef{planted 1-blossoming tree} $T$ is one of the two trees obtained after cutting a plain edge in a bicolored 1-blossoming tree, with the dangling half-edge of the cut-out edge as the \tdef{root}. It is red-planted (resp. blue-planted) if the root is incident to a red (resp. blue) node. Let $A \equiv A(x,y)$ (resp. $B \equiv B(x,y)$) be the counting series of red-planted (resp. blue-planted) bicolored 1-blossoming trees, with $x,y$ marking the numbers of blue and red nodes. As in~\Cref{fig:AB_decomposition}, a root-decomposition of planted 1-blossoming trees yields the system
 \[
 A=\frac{y}{(1-B)^2},\ \ \ B=\frac{x}{(1-A)^2},
 \]
 which is consistent with~\eqref{eq:defn-canopy-series} at $z=0$. Hence, $A=yF(A)$, with $F(u)=\frac{1}{(1-x/(1-u)^2)^2}$, so that the Lagrange inversion formula gives
 \begin{align*}
 [x^iy^j]A & = \frac1{j} [u^{j-1}][x^i]\left( \frac{1}{1-x/(1-u)^2} \right)^{2j}\\[.1cm]
 & = \frac1{j} [u^{j-1}]\binom{2j+i-1}{i}\frac1{(1-u)^{2i}} = \frac1{j}\binom{2j+i-1}{i}\binom{2i+j-2}{j-1}.
 \end{align*}

Every tree counted by $S_{i, j}$ has $i$ blue nodes and $j$ red nodes, thus has $2j+i-1$ corners at red nodes, among which $i+j-1$ are on the right of an edge, and $j$ are on the right of a bud. 
Adding a red dangling half-edge at such a corner, we get a tree counted by $[x^iy^j]A$. We thus have
 \[
 (2j + i - 1) S_{i, j} = [x^i y^j]A(x, y),
 \]
which gives the formula for $S_{i,j}$. A bijective derivation can be achieved by following the cyclic lemma approach from~\cite{chottin1975demonstration} for objects counted by $[x^iy^j]A$, combined with an encoding by integer compositions as in the proofs of~\Cref{prop:count_all_intervals} and~\Cref{prop:bostan}. 
\end{proof}

\begin{rem}
 The bicolored 1-blossoming trees are precisely those known~\cite{AlbenqueP15,fusy2007combinatoire} to bijectively encode rooted simple quadrangulations, or equivalently rooted non-separable maps. This is consistent with~\cite{fang17,fusy19}, which shows that rooted non-separable maps with $n + 1$ edges (resp. with $i+1$ vertices and $j+1$ faces)
 are in bijection with synchronized intervals counted by $S_n$ (resp. counted by $S_{i,j}$). It is also consistent with the fact that the above expression of $S_n$ (resp. $S_{i,j}$) corresponds to the known formula~\cite{nsp-refined,sc99,tutte1963census} for the number of rooted non-separable maps with $n+1$ edges (resp. with $i+1$ vertices and $j+1$ faces).
\end{rem}

\subsection{Modern intervals}\label{sec:co_mod}

As mentioned in~\Cref{rem:new-intervals}, the rise operator gives a bijection between modern intervals of size $n$ and \emph{new} intervals of size $n+1$. An explicit formula for the number of new intervals of size $n+1$ has been obtained in~\cite{chapoton06}, which we recover here bijectively via the enumeration of modern blossoming trees. 

\begin{prop} \label{prop:modern-counting}
 The number of modern intervals of size $n$ (also the number of new Tamari intervals of size $n+1$) is
 \[
 \frac{3 \cdot 2^{n-1}}{(n + 1)(n + 2)} \binom{2n}{n}.
 \]
\end{prop}
\begin{proof}
By~\Cref{lem:modern}, we have to show that the number of modern blossoming trees of size $n$ is given by the above formula. 
We define a \tdef{modern planted tree} to be one of the two components obtained by cutting a modern blossoming tree at the middle of a plain edge, rooted at the dangling half-edge. 
 As we do not account for the types of nodes in modern blossoming trees, we ignore colors in the following. 
 Let $A_\mm \equiv A_\mm(z)$ be the series of modern planted trees with $z$ marking the number of nodes, and $B_\mm \equiv B_\mm(z)$ the series of those with a bud immediately after the root half-edge in \cw-order. We recall from the definition that, in a modern blossoming tree, every plain edge can not be followed by a plain edge in \cw-order at both ends. This gives some restriction on the root-decomposition of modern planted trees. Indeed, for any sub-tree hanging from an edge that is not followed by a bud, it is to be counted by $B_\mm$. A tree counted by $A_\mm$ can be split into three (possibly empty) sequences of subtrees by the two buds of the root. Subtrees in the first sequence all follows a plain edge at the root, thus accounted by $B_\mm$, and also those in the second and the third one except the subtree leading the sequence, which follows a bud and thus has no restriction, and can be any tree accounted by $A_\mm$. The decomposition for trees counted by $B_\mm$ is similar, except that the third sequence is empty. By standard symbolic method, with $z$ accounting for the root, we have
 \begin{equation}
 \label{eq:modern-series}
 A_\mm = \frac{z}{1 - B_\mm} \left( 1 + \frac{A_\mm}{1 - B_\mm} \right)^2, \quad B_\mm = \frac{z}{1 - B_\mm} \left(1 + \frac{A_\mm}{1 - B_\mm}\right).
 \end{equation}

 We define $C_\mm = \frac{A_\mm}{1 - B_\mm}$. From~\eqref{eq:modern-series}, we observe that $A_\mm = B_\mm (1 + C_\mm)$. Substituting $A_\mm$ by $B_\mm (1 + C_\mm)$ in the definition of $C_\mm$ yields
 \[
 C_\mm= \frac{B_\mm}{1 - 2B_\mm}.
 \] 
 From the second equation in~\eqref{eq:modern-series}, we then obtain
 \[
 B_\mm = \frac{z}{1 - B_\mm} (1 + C_\mm) = \frac{z}{1 - 2B_\mm}.
 \]
Dividing the equation by $1-2B_\mm$, we get 
\[
C_\mm = z(1 + 2C_\mm)^2.
\]
 Hence, $C_\mm$ is also the generating series of complete binary trees with weight $z$ on internal nodes and weight $2$ on internal edges. As complete binary trees are counted by Catalan numbers, and there is exactly one less internal edge than internal nodes, we have
 \begin{equation}
 \label{eq:modern-cm}
 C_\mm(z) = \sum_{n \geq 1} \frac{2^{n-1}}{n+1} \binom{2n}{n} z^n.
 \end{equation}

 We now consider modern blossoming trees with one marked node. We may split such a tree at the two buds of the marked node, obtaining two sequences (distinguished by colors) of planted modern blossoming trees, which are formed in the same way as a sequence after a bud at the root in the arguments above to obtain~\Cref{eq:modern-series}. The number of modern blossoming trees of size $n$, which have $n + 1$ nodes each, is thus
 \begin{align*}
 \frac{1}{n + 1} [z^n] \left( 1 + \frac{A_\mm}{1 - B_\mm} \right)^2
 &= \frac{1}{n + 1} [z^n] \left( 2C_\mm + C_\mm^2 \right) = \frac{1}{n + 1} [z^n] \left( C_\mm + \frac{C_\mm - z}{4z} \right) \\
 &= \frac{1}{n + 1} \left( \frac{2^{n-1}}{n+1} \binom{2n}{n} + \frac{2^{n-2}}{n+2} \binom{2n+2}{n+1} \right) \\
 &= \frac{3 \cdot 2^{n-1}}{(n+1)(n+2)} \binom{2n}{n}.
 \end{align*}
\end{proof}

\begin{rem}
The proof above relies on combining identities on generating functions of families of planted trees. Based on similar arguments, one can derive an explicit bijection between objects counted by $B_\mm(z)$ and rooted plane trees with $n$ edges, each edge colored either black or white. We omit the details here, as the obtained bijection is not very enlightening. 
\end{rem}

\begin{rem}
With some more work, it should also be possible to extend the decomposition of modern planted trees to track the numbers of nodes of the three possible canopy-types, and to compute the trivariate generating function $F_\mm(x,y,z)$ of modern intervals, with $x,y,z$ marking respectively canopy-entries $\canoo, \canzz, \canoz$. It is known~\cite{fang2021bijective} that $zF_\mm(x,y,z)$ counts rooted bipartite maps by the numbers of black vertices, white vertices and faces, thus $zF_\mm(x,y,z)$ is symmetric in $x, y, z$. We have not been able to directly see this symmetry at the level of modern bicolored blossoming trees. It would also be interesting to find a closure-bijection from these trees to rooted bipartite maps. 
\end{rem}

\subsection{Modern-synchronized intervals} \label{sec:sp_modsyn}

\begin{figure}
\begin{center}
\includegraphics[width=\linewidth]{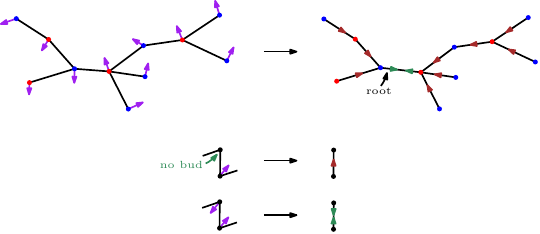}
\end{center}
\caption{A modern bicolored 1-blossoming tree. The result of transferring the bud information into an orientation of the half-edges yields a tree with one bidirected edge $e$ and 
all other edges unidirected toward the closest extremity of $e$.}
\label{fig:modern_synchronized}
\end{figure}

We show here that the intersection of the families of modern and synchronized intervals yields a Catalan family. A bicolored 1-blossoming tree $T$ of size $n$ is called \tdef{modern} if it is also the reduction of a modern synchronized blossoming tree.

\begin{prop}
 For $n \geq 1$, the number of modern-synchronized Tamari intervals of size $n$ is the $n$-th Catalan number. For $i, j \geq 1$, the number of modern-synchronized Tamari intervals with $i$ canopy-entries $\canoo$ and $j$ canopy-entries $\canzz$ is given by the Narayana number
 \[
 \frac{1}{i+j-1} \binom{i+j-1}{i} \binom{i+j-1}{j}.
 \]
\end{prop}
\begin{proof}
 By~\Cref{coro:sync} and~\Cref{lem:modern}, we count instead bicolored 1-blossoming trees. For each bud $b$ in such a tree $T$, let $v$ be its incident node, and $h$ the half-edge that follows $b$ in \ccw-order around $v$. We give to $h$ an orientation away from $v$. See~\Cref{fig:modern_synchronized} for an example. By the absence of non-modern edges, each edge has at least one oriented half-edge, and when there are two, they are in head-to-head direction. In the first case, we endow the whole edge with the same orientation, and in the second case, we say that the edge is bi-oriented.

 Now, deleting the buds, we obtain an encoding of $T$ as a vertex-bicolored unrooted plane tree with an orientation on its $n$ edges. Since there are $n + 1$ vertices, thus $n + 1$ buds and $n + 1$ oriented half-edges, there must be a single bi-oriented edge $e$. We also observe that every vertex has outdegree~$1$ on its half-edges. By a simple induction on the distance towards $e$, all edges other than $e$ are oriented toward $e$, see~\Cref{fig:modern_synchronized}. We may then further encode such an oriented tree by rooting $T$ at the half-edge incident to the blue end of the bi-oriented edge, leading to a rooted plane tree with its vertex coloring inherited from the $1$-blossoming tree structure, with the root-vertex blue. However, the color of the root vertex determines the color of other vertices. Therefore, the final encoding is simply a rooted plane tree, thus counted by Catalan numbers.

 For the refined enumeration, using~\Cref{coro:triple} and the definition of $1$-blossoming trees, by the same arguments above, it amounts to counting bicolored rooted plane trees with $i$ blue nodes and $j$ red nodes. As the root is blue, this is given by the Narayana number \cite{narayana-even-odd}.
\end{proof}

\begin{rem}
The fact that rooted plane trees having $i$ nodes at even depth and $j$ nodes at odd depth are counted by the Narayana number follows from the fact that they are in bijection with rooted plane trees having $i$ inner nodes and $j$ leaves, see~\cite[Sec.3]{janson2015scaling} and~\cite{narayana-bijection}. 
One can also associate to a binary tree $T$ the rooted plane tree whose alternating layout is the smooth drawing of $T$. This gives a bijection from binary trees of size $n$ to rooted plane trees with $n$ edges, which maps the two types of leaves (left and right) in binary trees to the two types of nodes (even and odd) in rooted plane trees.
\end{rem}

\begin{rem}
This result is also a consequence of~\cite{fang2021bijective}. Indeed, the rise operator applied to a modern interval increases the number of entries $\canoz$ by $1$, while preserving the numbers of entries of type $\canoo$ and $\canzz$. Hence, by \ref{rem:new-intervals}, the number of modern-synchronized intervals with $i$ canopy-entries $\canoo$ and $j$ canopy-entries $\canzz$ is the number of new Tamari intervals with $i$ canopy-entries of type $\canoo$, $j$ of type $\canzz$, and one of type $\canoz$. It follows from~\cite{fang2021bijective} that it is also the number of rooted bipartite planar maps with a unique face (i.e., rooted plane trees), with $i$ black vertices and $j$ white vertices in the proper 2-coloring with the root-vertex black.
\end{rem}

\begin{rem}\label{rk:trivial_mirror}
The structure of the orientation of $T$ obtained from the buds clearly implies that, if a blossoming tree is modern and synchronized, then it is infinitely modern. We also note that these blossoming trees are the mirror of the blossoming trees for trivial intervals (\Cref{rem:trivial} and~\Cref{fig:trivial}). 
\end{rem}

\subsection{Kreweras and infinitely modern intervals}\label{sec:inf_mo}

Via the bijection between Kreweras intervals of size $n$ and non-crossing trees with $n$ edges (\Cref{rk:non-crossing-tree}), which are well-known to be in bijection to ternary trees with $n$ nodes, we recover the following known result~\cite{bernardi09,Kre72}.

\begin{prop}\label{prop:count_kre}
 The number of Kreweras intervals of size $n$ is
 \[
 \frac{1}{2n+1}\binom{3n}{n}.
 \]
\end{prop}

Moreover, as noted in Remark~\ref{rem:inf_modern_mirror}, the blossoming trees for infinitely modern intervals are just the reflection of the blossoming trees for Kreweras intervals, which induces a bijection between these two interval families. We thus recover the following result from~\cite{rognerud18}.

\begin{prop}\label{prop:count_inf_modern}
 The number of infinitely modern Tamari intervals of size $n$ is
 \[
 \frac{1}{2n+1}\binom{3n}{n}.
 \]
\end{prop}
 
\subsection{Self-dual intervals}\label{sec:self_dual_counting}

We use here~\Cref{lem:duality-commutation} to bijectively obtain counting formulas for self-dual intervals of size $n$ in all families considered so far.

\begin{table}
 \centering
 \begin{tabular}{cccc}
 \toprule
 \textbf{Types} & \hspace{.4cm}\begin{tabular}{c}General \\ size $n$\end{tabular}\hspace{.4cm} & \begin{tabular}{c}Self-dual \\ size $2k$\end{tabular} & \hspace{.4cm}\begin{tabular}{c}Self-dual \\ size $2k + 1$\end{tabular}\hspace{.4cm} \\
 \midrule
 \textbf{General} & $\displaystyle \frac{2}{n(n+1)} \binom{4n+1}{n-1}$ & $\displaystyle \frac{1}{3k+1} \binom{4k}{k}$ & $\displaystyle \frac1{k+1} \binom{4k+2}{k}$ \\ \midrule
 \textbf{Synchronized} & $\displaystyle \frac{2}{n(n+1)}\binom{3n}{n-1}$ & $\displaystyle 0$ & $\displaystyle \frac1{k+1}\binom{3k+1}{k}$ \\ \midrule
 \textbf{\begin{tabular}{c}Modern \\ / new for size-1\end{tabular}} & $\displaystyle \frac{3\cdot 2^{n-1}}{(n+1)(n+2)} \binom{2n}{n}$ & $\displaystyle \frac{2^{k-1}}{k+1} \binom{2k}{k}$ & $\displaystyle \frac{2^k}{k+1} \binom{2k}{k}$ \\ \midrule
 \textbf{\begin{tabular}{c}Modern and \\ synchronized\end{tabular}} & $\displaystyle \frac{1}{n+1} \binom{2n}{n}$ & $\displaystyle 0$ & $\displaystyle \frac1{k+1} \binom{2k}{k}$ \\ \midrule
 \textbf{\begin{tabular}{c}Inf. modern \\ / Kreweras\end{tabular}} & $\displaystyle \frac1{2n+1}\binom{3n}{n}$ & $\displaystyle \frac1{2k+1} \binom{3k}{k}$ & $\displaystyle \frac1{k+1} \binom{3k+1}{k}$ \\
 \bottomrule
 \end{tabular}
\caption{Counting formulas for general and self-dual Tamari intervals.}
\label{table:self_dual}
\end{table}

\begin{prop}\label{prop:self-dual}
 The number of self-dual Tamari intervals of size $n$ is given by the formulas in~\Cref{table:self_dual} (depending on the parity of $n$) for each of the following families: general, synchronized, modern, new, modern and synchronized, infinitely modern, and Kreweras.
\end{prop}
\begin{proof} 
 By~\Cref{lem:duality-commutation}, the number of self-dual Tamari intervals of size $n$ is the number of trees in $\blossoming_n$ that are invariant by switching colors of half-edges. They are also the uncolored blossoming trees with a half-turn symmetry. Now we discuss the case for each family. The parity of $n$ matters, as it changes the quotient of related trees by the half-turn symmetry. In general, the center of rotation is a node when $n$ is even; otherwise, it is an edge.
\smallskip

\noindent\textbf{Self-dual general Tamari intervals.} For $n = 2k$ even, the center of rotation of such a tree is a node $u$, and the quotient of the tree by the half-turn symmetry gives a blossoming tree $T$ of size $k$ with a marked synchronized node $u$. By an argument similar to that in the proof of~\Cref{prop:count_inf_modern}, we can decompose $T$ into four parts: the sub-trees of $u$ except its leftmost descending edge $e = \{ u, v \}$, and the three sequences of sub-trees of $v$ separated by its two buds. We thus have a size-preserving recursive bijection between blossoming trees rooted at a synchronized node and rooted 4-ary trees. Hence, self-dual Tamari intervals of size $n = 2k$ are counted by $\frac1{3k+1}\binom{4k}{k}$.

 For $n = 2k + 1$ odd, the center of rotation is a plain edge, and the quotient of the tree by the half-turn symmetry gives a planted blossoming tree with $k$ nodes, defined in the proof of \ref{prop:trivariate}. As seen in the proof of \ref{prop:trivariate}, when taking $x = y = z = t$, the counting series $A\equiv A(t)$ of these trees satisfies $A = t/(1-A)^3$, and we get $[t^{k+1}]A = \frac1{k+1} \binom{4k+2}{k}$ using Lagrange inversion or the cyclic lemma.

\smallskip

\noindent \textbf{Self-dual synchronized intervals.} In this case, we note that there is no self-dual blossoming tree of even size, as the center of rotation would be a node that is necessarily not synchronized. For $n = 2k + 1$ odd, the center of rotation is a plain edge, and the quotient tree is a planted blossoming tree with $k$ nodes that are all synchronized, i.e., the two incident buds of each node are consecutive. They thus split the subtrees of the root into two sequences. The counting series $A \equiv A(t)$ of these trees satisfies $A = t/(1-A)^2$, and the standard techniques give $[t^{k+1}]A = \frac1{k+1} \binom{3k+1}{k}$.

\smallskip

\noindent \textbf{Self-dual modern/new intervals.} We use the series $A_\mm, B_\mm, C_\mm$ defined in the proof of~\Cref{prop:modern-counting}. For $n = 2k$ even, the quotient trees are modern blossoming trees rooted at a synchronized node, with the buds at the top. For a root decomposition, as the blossoming tree involved is modern, we see that all subtrees except the rightmost one must be those counted by the series $B_\mm$, that is, those with a bud next to the root in \cw-order. With symbolic method, by the definition of $C_\mm$ and~\Cref{eq:modern-cm}, the number of self-dual modern intervals is
 \[
 [z^k] A_\mm \frac1{1 - B_\mm} = [z^k]C_\mm = \frac{2^{k-1}}{k+1} \binom{2k}{k}.
 \]
 
 For $n = 2k+1$ odd, the center of symmetry is an edge $e$. For $e$ to not be a non-modern edge, there must be a bud that follows $e$ in \cw-order at both ends, meaning that the quotient trees are exactly those counted by $B_\mm$. The number of self-dual modern intervals is thus
 \[
 [z^{k}]B_\mm = \frac{2^k}{k+1} \binom{2k}{k}.
 \]
 We observe that this is exactly twice the number for $n = 2k$.

 Regarding new intervals, as the rise operator preserves the property of being-dual, the number of self-dual new intervals of size $n$ equals the numbers of self-dual modern intervals of size $n-1$.
\smallskip

\noindent \textbf{Self-dual modern and synchronized intervals.} In this case, the encoding of related blossoming tree by a rooted plane tree illustrated in~\Cref{fig:modern_synchronized} commutes with the half-turn rotation, i.e., self-dual blossoming trees correspond to plane trees with a marked edge and invariant by a half-turn rotation. Such trees with $2k + 1$ edges are clearly counted by the Catalan numbers $\mathrm{Cat}_k$, as the quotient trees are plane trees with $k$ edges and an additional dangling half-edge as the root. The trees with size $n = 2k$ is clearly $0$ by the case of synchronized intervals.
\smallskip

\noindent \textbf{Self-dual Kreweras and infinitely modern intervals.} 
Via the bijection with non-crossing trees, self-dual Kreweras intervals correspond to non-crossing trees that are fixed by left-right mirror, this is indeed equivalent to the fact that the associated meandering tree is stable by half-turn. It is then an easy exercise to count these non-crossing trees of size $n$. For $n=2k$, such a tree is obtained as a non-crossing tree $\tau$ of size $k$ concatenated with its left-right mirror $\overline{\tau}$, the right end of $\tau$ being merged with the left end of $\overline{\tau}$. Hence, the number of self-dual Kreweras intervals of size $2k$ is equal to the number of non-crossing trees of size~$k$. 
For $n=2k+1$, a non-crossing tree of size $n$ fixed by left-right mirror is obtained as follows. Take a pair $(\tau_1,\tau_2)$ of non-crossing trees whose sizes add up to $k$, concatenate $\tau_1$ and $\tau_2$ into a non-crossing tree $\tau$, with $v$ the vertex resulting from merging the right end of $\tau_1$ with the left end of $\tau_2$. Then concatenate $\tau$ with its left-right mirror $\overline{\tau}$ without merging the ends, and add an edge from $v$ to the corresponding vertex $\overline{v}$ in $\overline{\tau}$. From this construction, the number of self-dual Kreweras intervals of size $2k+1$ is $[z^k](1+R(z))^2$, with $R(z)=z(1+R(z))^3$, which gives the formula. 

 Finally, by Remark~\ref{rem:inf_modern_mirror}, the infinitely modern blossoming trees are the reflection of the Kreweras blossoming trees. 
 Note that a blossoming tree is fixed by a half-turn if and only if its reflection is also fixed by a half-turn. Hence, the induced bijection between Kreweras intervals and infinitely modern intervals preserves the property of being self-dual, so that in each size the numbers of self-dual intervals are the same in both families. 
\end{proof}

\begin{rem}
 In the context of non-crossing partitions, with the notation $\overline{\pi}$ for the Kreweras complement of a partition, there is a natural duality, which maps an interval $(\pi,\pi')$ to $(\overline{\pi'},\overline{\pi})$. Via the bijection between $NC_n$ and $\bintree_n$ this duality is consistent with the Tamari duality, as follows from the property illustrated in~\Cref{fig:complement_bij}. 
 The formula in~\Cref{table:self_dual} for self-dual Kreweras intervals has been previously obtained, see~\href{https://oeis.org/A047749}{OEIS A047749}. 
\end{rem}

\section{Final remarks}
\label{sec:final}

\subsection{Dyck walks}\label{sec:dyck}

The Tamari lattice is often presented as a poset on Dyck walks. This point of view has certain advantages, for instance to formulate recursive decompositions~\cite{bousquet2023intervals,bousquet2013representation,bousquet2011number,fang17}. 
Duality on the other hand is not as obvious as left-right symmetry of trees.

The bijection of intervals with meandering trees is easy to characterize on Dyck walks, via the underlying correspondence to binary trees (a binary tree $T=(T_1,T_2)$ is mapped to the Dyck walk $D=D_1\!\nearrow\!D_2\!\searrow$, with $D_1,D_2$ the Dyck walks associated inductively to $T_1,T_2$).  
Recall the contact-vector $C(W)=(c_0,\ldots, c_n)$ and descent-vector $D(W)=(d_0, \ldots, d_n)$ attached to a Dyck walk $W$ of length $2n$. That is, $c_i$ is the number of contacts after the $i$th up step of $W$ for $i>0$, while $c_0$ is the number of contacts of $W$. Thus $c_i=0$ if and only if the $i$-th up step is followed by a down step (this happens in particular when $i=n$). On the other hand, $d_i$ is the number of down steps after the $i$th up step of $W$, with $d_0 = 0$ by convention. For $W$ a Dyck walk associated to a binary tree $T$, it is easy to check that the degree-vector of $T$ is $C(W)$, while the degree vector of $\mir(T)$ is $D(W)$ read from right to left. This leads to the following, see~\Cref{fig:dyck}.

\begin{prop}\label{prop:dyck}
In the Dyck walk formulation, a Tamari interval $(W, W')$ corresponds to the meandering tree $M\in \mtree_n$ such that $C(W')$ is the degree-vector of the upper diagram-drawing, while $D(W)$, read from right to left, is the degree-vector of the lower diagram-drawing rotated by a half-turn.   
\end{prop}

\begin{figure}[!ht]
\begin{center}
\includegraphics[width=0.7\textwidth]{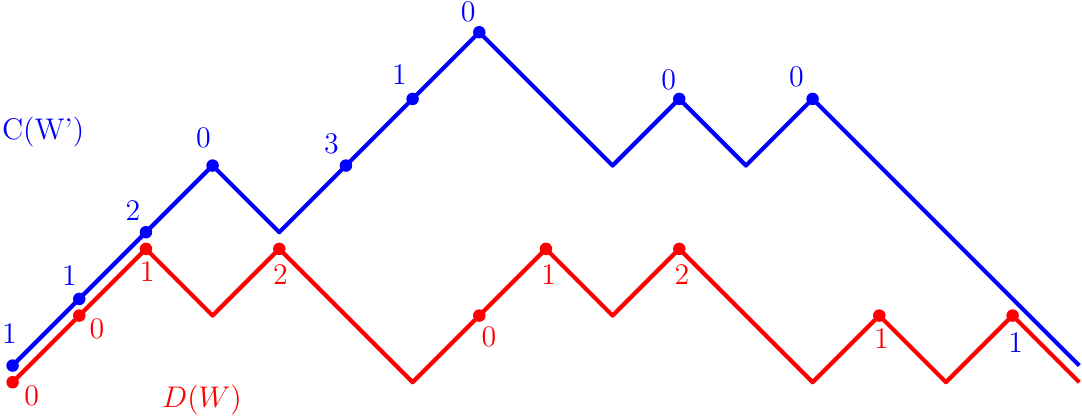}
\end{center}
\caption{\label{fig:dyck}
Formulated on Dyck walks, the interval corresponding to the meandering tree in ~\Cref{fig:lengths}.
}
\end{figure}

The recursive decomposition of intervals is then easily translated in our model (this can also be done starting with interval posets as in~\cite{ChatelPons15}). Given a meandering tree, consider the upper arc $(0,i+1/2)$ with $i$ maximal; let $(i+1/2,j)$ be the corresponding lower arc. Deleting this arc, we get a meandering tree on $[0,i]$, a meandering tree on $[i+1,n]$ and a point $j$ such that no lower arc encloses it. Conversely, such data corresponds to a meandering tree of size $n$.

\subsection{Limitations}

We are currently unable to use our bijection to count $m$-Tamari intervals, which are synchronized intervals with canopy of the form \[1,0^m,1,0^m,\ldots,1,0^m.\] Indeed the order on canopy entries is lost in the bijection.

 By~\Cref{lem:bi-lengths} we can count Tamari intervals with respect to the unordered bi-degree profile. From~\Cref{prop:dyck}, when formulated on pairs of Dyck walks, this gives enumeration with respect to the unordered joint profile of the descent-vector for lower walk and level-vector for upper walk. This, however, does not lead to counting labeled intervals from~\cite{bousquet2013representation}, for which we would need to have control on the ascent-vector of the upper walk, or equivalently on the descent-vector of the upper walk via the involution in~\cite{Pons19}. 

\def\refl{\mathrm{refl}}
\def\invint{\rho}

\subsection{A new involution on Tamari intervals}\label{sec:invol} 
Previously known involutions on Tamari intervals are the classical duality involution, and more recently the involution in~\cite{Pons19}. The mirror of blossoming trees allows us to provide a new involution. 

More formally, we define the \tdef{reflection} of a blossoming tree $B \in \blossoming_n$, denoted by $\refl(B)$, to be the mirror image of $B$, and define $\invint = \blomtoint \circ \refl \circ \inttoblom$, which is clearly an involution on Tamari intervals. 

\begin{figure}
\begin{center}
\includegraphics[width=10cm]{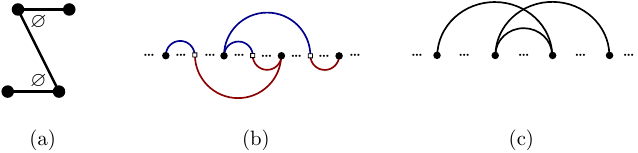}
\end{center}
\caption{(a) Forbidden pattern in blossoming trees that are the reflection of modern blossoming trees. (b) The corresponding configuration in the meandering representation of the blossoming tree. (c) The corresponding configuration in the interval-poset tree.}
\label{fig:mirr_modern}
\end{figure}

\begin{prop}
The involution $\invint$ commutes with the duality involution $\mir$. It preserves the property of being synchronized, matches the infinitely modern intervals with the Kreweras intervals, 
and matches the modern and synchronized intervals with the trivial intervals. 

Furthermore, it matches the modern intervals with the intervals whose interval-poset tree has no triple of arcs as in~\Cref{fig:local_op_int_poset_tree}(c), or equivalently those whose interval-poset has no triple $x,y,z$ of elements such that $\mathrm{Int}(x)=\mathrm{Int}(y)\cap\mathrm{Int}(z)$.
\end{prop}
\begin{proof}
The first statement follows from the fact that the operations of reflection and of color-switch on bicolored blossoming trees commute. 
Clearly the reflection does not affect the property that the two buds at each node are grouped. It thus preserves the property of being synchronized. 
From~\Cref{rem:inf_modern_mirror} the infinitely modern intervals are matched with the Kreweras intervals. 
From~\Cref{rk:trivial_mirror} the modern synchronized intervals are matched with the trivial intervals. 

Finally, from~\Cref{lem:modern}, modern intervals are matched by $\invint$ with intervals whose blossoming trees have no plain edge followed by a plain edge at both ends in \ccw-order. In the meandering representation, due to the absence of flawed pairs, the configuration for the three plain edges is as in~\Cref{fig:mirr_modern}(b). Thus, the interval-poset tree, obtained by performing the operation of~\Cref{fig:local_op_int_poset_tree} from right to left, has to avoid the pattern in~\Cref{fig:mirr_modern}(c).
\end{proof}

\subsection{Self-dual intervals and $q$-analogues} 
 Regarding the counting formulas in~\Cref{table:self_dual}, it has been observed by Vic Reiner (personal communication) 
 that the number of self-dual intervals coincides with a simple $q$-analogue of the formula for all intervals taken at $q=-1$. We have checked that the same holds for synchronized intervals. 
 It would be nice to have a natural explanation of this fact. This may come from a combinatorial analysis of blossoming trees.
 
 \begin{figure}
 \begin{center}
 \raisebox{1.44cm}{\includegraphics[width=6cm]{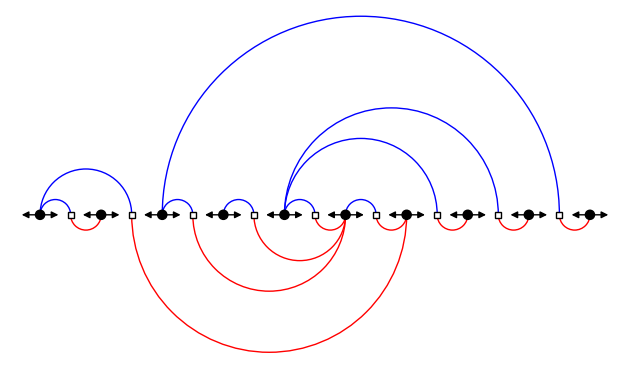}}\hspace{.5cm} \raisebox{.05cm}{\includegraphics[width=6cm]{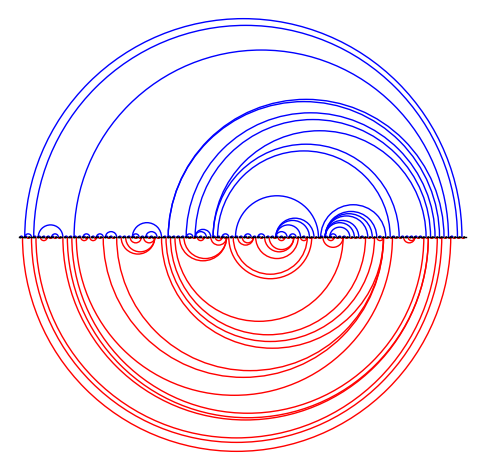}} \includegraphics[width=6cm]{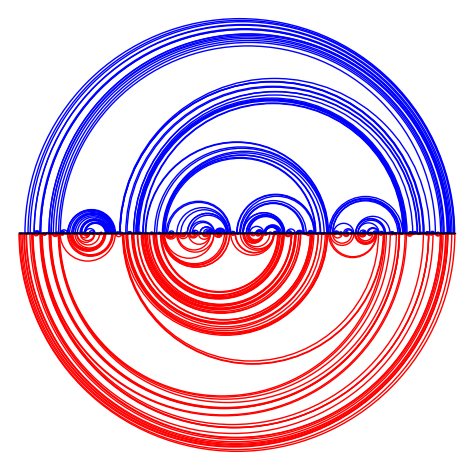}\hspace{.5cm}\includegraphics[width=6cm]{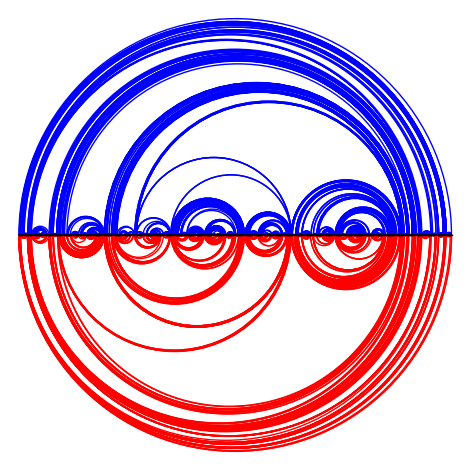}
 \end{center}
 \caption{Random bicolored blossoming trees in their meandering representation, with $10$, $100$, $1000$, $10000$ nodes respectively. In large size, they look similar to the smooth drawing of the corresponding Tamari interval.}
 \label{fig:simulations}
 \end{figure}
 
 \subsection{Implementation} 
An implementation of the bijection is available at the link \url{https://github.com/fwjmath/assorted-tamari/blob/master/blossoming.py}\\
which also includes a random generator for Tamari intervals. Some random samples are shown in~\Cref{fig:simulations}.

\medskip
\medskip

\noindent{\emph{Acknowledgement.} The authors are grateful to Fr\'ed\'eric Chapoton, Vincent Pilaud, Vic Reiner, and Gilles Schaeffer for interesting discussions. The first author is partially supported by ANR-21-CE48-0007 (IsOMA) and ANR-21-CE48-0020 (PAGCAP). The second and third authors are partially supported by ANR-19-CE48-0011 (COMBIN\'E).

\bibliographystyle{alpha}
\bibliography{biblio}

\end{document}